\documentclass[11pt]{amsart}
\usepackage{amssymb}
\usepackage{amsmath}
\usepackage{amsfonts}
\usepackage{amscd}
\usepackage[mathscr]{eucal}
\usepackage{graphicx}
\usepackage{float}
\newtheorem{Theorem}{Theorem}[section]
\newtheorem{Corollary}[Theorem]{Corollary}
\newtheorem{Lemma}[Theorem]{Lemma}
\newtheorem{Proposition}[Theorem]{Proposition}
\newtheorem*{Prop3.2}{Proposition~3.2}
\newtheorem*{Thm3.3}{Theorem~3.3}
\newtheorem*{Prop4.2}{Proposition~4.2}
\newtheorem*{Cor4.3}{Corollary~4.3}
\newtheorem*{Lem5.7}{Lemma~5.7}
\theoremstyle{definition}
\newtheorem{Definition}[Theorem]{Definition}
\newtheorem{Example}[Theorem]{Example}

\newtheorem*{Problem}{Problem}
\newtheorem*{ReconstructionConjecture}{Reconstruction~Conjecture}
\theoremstyle{remark}
\newtheorem*{Remark}{Remark}

\begin{document}
\sloppy
\title{The Reconstruction Conjecture for finite simple graphs and associated directed graphs}
\author{Tetsuya Hosaka}
\email{hosaka.tetsuya@shizuoka.ac.jp}
\address{
Department of Mathematics, Faculty of Science, Shizuoka University, 836 Ohya, Suruga-ku, Shizuoka, 422-8529, Japan}
\date{June 18, 2021}

\begin{abstract}
In this paper, we study the Reconstruction Conjecture for finite simple graphs.
Let $\Gamma$ and $\Gamma'$ be finite simple graphs with at least three vertices 
such that there exists a bijective map $f:V(\Gamma) \rightarrow V(\Gamma')$ and 
for any $v\in V(\Gamma)$, there exists an isomorphism $\phi_v:\Gamma-v \to \Gamma'-f(v)$.
Then we define the associated directed graph 
$\widetilde{\Gamma}=\widetilde{\Gamma}(\Gamma,\Gamma',f,\{\phi_v\}_{v\in V(\Gamma)})$ 
with two kinds of arrows from the graphs $\Gamma$ and $\Gamma'$, 
the bijective map $f$ and the isomorphisms $\{\phi_v\}_{v\in V(\Gamma)}$.
By investigating the associated directed graph $\widetilde{\Gamma}$, 
we study when are the two graphs $\Gamma$ and $\Gamma'$ isomorphic.
\end{abstract}

\maketitle

\section{Introduction and preliminaries}\label{sec1}

A finite simple graph $\Gamma$ is said to be {\it reconstructible}, 
if any simple graph $\Gamma'$ with the following property~$(*)$ 
is isomorphic to $\Gamma$.
\begin{enumerate}
\item[$(*)$] There exists a bijective map $f:V(\Gamma) \rightarrow V(\Gamma')$ such that 
for any $v\in V(\Gamma)$, there exists an isomorphism $\phi_v:\Gamma-v \to \Gamma'-f(v)$.
\end{enumerate}
Here $V(\Gamma)$ and $V(\Gamma')$ are the vertex sets of $\Gamma$ and $\Gamma'$ respectively.
Also $\Gamma-v$ and $\Gamma'-f(v)$ 
are the full subgraphs of $\Gamma$ and $\Gamma'$ whose vertex sets are 
$V(\Gamma)-\{v\}$ and $V(\Gamma')-\{f(v)\}$ respectively.

The Reconstruction Conjecture (also often called Ulam Conjecture or Kelly-Ulam Conjecture) 
introduced by P.J.~Kelly and S.M.~Ulam in 1942 
is one of the very famous open problems in Graph Theory and in Mathematics.

\begin{ReconstructionConjecture}
Every finite simple graph with at least three vertices will be reconstructible.
\end{ReconstructionConjecture}

Some classes of reconstructible graphs are known 
(see \cite{BKY}, \cite{BoHe}, \cite{FL1}, \cite{Ho}, \cite{HoX}, 
\cite{Ke}, \cite{Ke2}, \cite{L}, \cite{Ma}, \cite{Mc}, \cite{Mc2}, \cite{Ni} and \cite{U}) 
as follows:
Let $\Gamma$ be a finite simple graph with at least three vertices.
\begin{enumerate}
\item[(i)] If $\Gamma$ is a regular graph, then it is reconstructible.
\item[(ii)] If $\Gamma$ is a tree, then it is reconstructible.
\item[(iii)] If $\Gamma$ is not connected, then it is reconstructible.
\item[(iv)] If $\Gamma$ has at most 11 vertices, then it is reconstructible.
\item[(v)] If $\Gamma$ is a maximal planar graph, then it is reconstructible.
\item[(vi)] If $\Gamma$ is a finite graph that is the 1-skeleton 
of some simplicial flag complex that is a homology manifold of dimension $n \ge 1$, 
then it is reconstructible.
\end{enumerate}
Also it is known that if all 2-connected graphs are reconstructible, then 
every finite simple graph is reconstructible \cite{Y}.

Let $\Gamma$ and $\Gamma'$ be finite simple graphs with at least three vertices 
satisfying the property~$(*)$ as above.

Then we define the directed graph 
\[ \widetilde{\Gamma}=\widetilde{\Gamma}(\Gamma,\Gamma',f,\{\phi_v\}_{v\in V(\Gamma)}) \]
with two kinds of arrows 
(called \textit{the associated directed graph}) as follows:
\begin{enumerate}
\item[(1)] The vertex set $V(\widetilde{\Gamma})$ of $\widetilde{\Gamma}$ is $V(\Gamma)$.
\item[(2)] For $v_1, v_2 \in V(\widetilde{\Gamma})$, 
there is a normal-arrow from $v_1$ to $v_2$ in $\widetilde{\Gamma}$, 
denoted by $v_1 \longrightarrow v_2$, 
if and only if $[v_1,v_2]\not\in E(\Gamma)$ and $[f(v_1), \phi_{v_1}(v_2)]\in E(\Gamma')$.
\item[(3)] For $v_1, v_2 \in V(\widetilde{\Gamma})$, 
there is a dashed-arrow from $v_1$ to $v_2$ in $\widetilde{\Gamma}$, 
denoted by $v_1 \dashrightarrow v_2$, 
if and only if $[f(v_1),f(v_2)]\not\in E(\Gamma')$ and $[v_1, \phi_{v_1}^{-1}(f(v_2))]\in E(\Gamma)$
\end{enumerate}
(see Figure~\ref{figure0-definition1}).
\begin{figure}[h]
\centering
{
\vspace*{2mm}
\includegraphics[keepaspectratio, scale=0.90, bb=0 0 386 130]{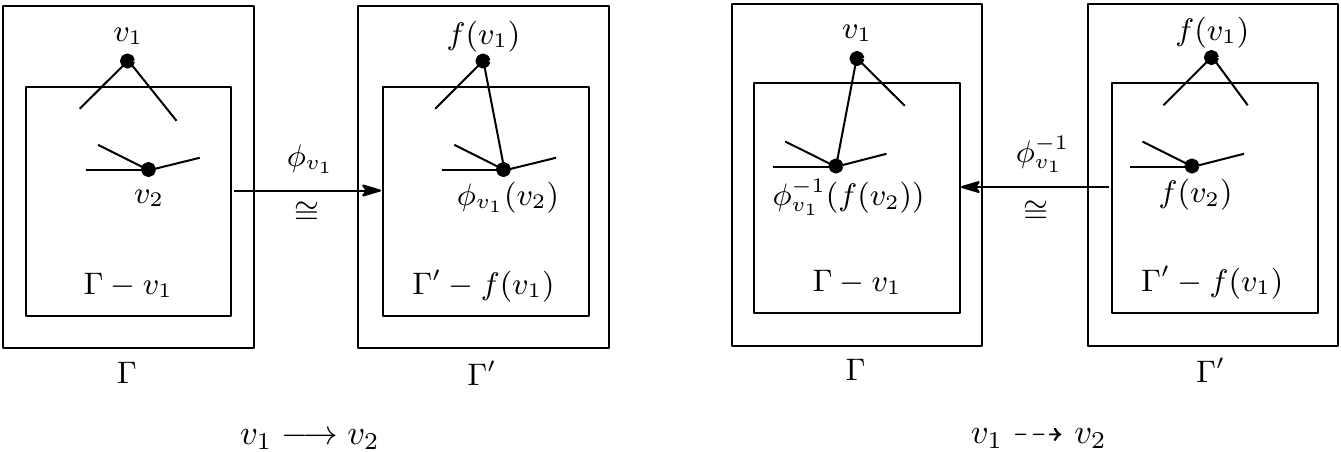}
}
\caption{Definition of arrows in $\widetilde{\Gamma}$}
\label{figure0-definition1}
\end{figure}

The purpose of this paper is to consider when are the two graphs $\Gamma$ and $\Gamma'$ isomorphic 
by investigating the associated directed graph $\widetilde{\Gamma}$, 
and we study the problem of when are finite simple graphs reconstructible.

We give some examples of $\widetilde{\Gamma}$ in Section~\ref{sec2} 
and we investigate some properties of $\widetilde{\Gamma}$ in Section~\ref{sec3}.

\begin{Prop3.2}
For all $v\in V(\widetilde{\Gamma})$, 
the number of normal-arrows arising from $v$ and 
the number of dashed-arrows arising from $v$ are equal in $\widetilde{\Gamma}$.
\end{Prop3.2}

We show the following theorem in Section~\ref{sec3}.

\begin{Thm3.3}
For any vertex $v \in V(\Gamma)=V(\widetilde{\Gamma})$, 
the following two statements are equivalent$:$
\begin{enumerate}
\item[\textnormal{(i)}] 
There are no normal-arrows arising from $v$ (and hence no dashed-arrows arising from $v$) in 
the associated directed graph $\widetilde{\Gamma}$.
\item[\textnormal{(ii)}]
The isomorphism $\phi_{v}:\Gamma-v \to \Gamma'-f(v)$ 
extends to the isomorphism $\overline{\phi_{v}}:\Gamma \to \Gamma'$ 
as $\overline{\phi_{v}}|_{\Gamma-v}=\phi_{v}$ and $\overline{\phi_{v}}(v)=f(v)$.
\end{enumerate}
\end{Thm3.3}

By Theorem~\ref{Thm1}, 
if for some vertex $v \in V(\widetilde{\Gamma})$,
\begin{enumerate}
\item[\textnormal{(i)}] 
there are no normal-arrows arising from $v$ (and no dashed-arrows arising from $v$) in 
the associated directed graph $V(\widetilde{\Gamma})$, 
\end{enumerate}
then $\Gamma\cong \Gamma'$.

\medskip

Now we state the otherwise case.
Suppose that for every vertex $v \in V(\widetilde{\Gamma})$ 
the statement (i) above does not hold.
Then there exist a normal-arrow and a dashed-arrow arising 
from each vertex $v \in V(\widetilde{\Gamma})$ 
by Proposition~\ref{Proposition1}.
Since the directed graph $\widetilde{\Gamma}$ is finite, 
for some positive number $k\in {\mathbb{N}}$ and 
some vertices $v_1,\ldots, v_{2k} \in V(\widetilde{\Gamma})$, 
there exist arrows $v_{2i-1} \longrightarrow v_{2i}$ and $v_{2i} \dashrightarrow v_{2i+1}$ 
in $\widetilde{\Gamma}$ 
for any $i=1,\ldots,k$, where $v_{2k+1}:=v_1$ (see Figure~\ref{figure-cycle}).
Here the vertices $v_1,\ldots, v_{2k}$ need not be different all together.
\begin{figure}[h]
\centering
{
\includegraphics[keepaspectratio, scale=0.90, bb= 0 0 373 64]{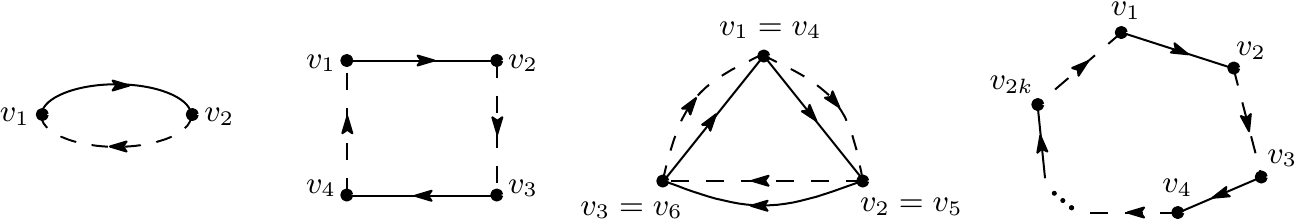}
}
\caption{Examples of cycles with alternate normal-arrows and dashed-arrows}\label{figure-cycle}
\end{figure}

We investigate on cycles with alternate normal-arrows and dashed-arrows in $\widetilde{\Gamma}$ in Section~\ref{sec4}.
For such a cycle, we define ``$\alpha$-type'' and ``$\beta$-type''.
We show that if such a cycle is $\alpha$-type then 
there exist normal-arrows $v_1 \longrightarrow v_2$ and 
$v_2 \longrightarrow v_1$ in $\widetilde{\Gamma}$ in Section~\ref{sec4}.

\begin{Prop4.2}
If for some positive number $k\in {\mathbb{N}}$ and 
some sequence $v_1,\ldots,v_{2k} \in V(\widetilde{\Gamma})$, 
there exist arrows $v_{2i-1} \longrightarrow v_{2i}$ and 
$v_{2i} \dashrightarrow v_{2i+1}$ in $\widetilde{\Gamma}$ 
for any $i=1,\ldots,k$ (where $v_{2k+1}:=v_1$) 
and if the cycle $v_1,\ldots,v_{2k}$ is $\alpha$-type 
then 
there exist normal-arrows $v_1 \longrightarrow v_2$ and 
$v_2 \longrightarrow v_1$ in $\widetilde{\Gamma}$.
\end{Prop4.2}

By Proposition~4.2, we obtain the following.

\begin{Cor4.3}
Let $\Gamma$ and $\Gamma'$ be finite simple graphs with at least three vertices 
satisfying the property~$(*)$.
Then for the associated directed graph $\widetilde{\Gamma}$, 
\begin{enumerate}
\item[\textnormal{(I)}] there exists a vertex $v\in V(\widetilde{\Gamma})$ 
with no normal-arrow arising from $v$ in $\widetilde{\Gamma}$ 
(then $\Gamma \cong \Gamma'$ by Theorem~\ref{Thm1}), 
\item[\textnormal{(II)}] for some $v_1,v_2\in V(\widetilde{\Gamma})$,
there exist normal-arrows $v_1 \longrightarrow v_2$ and $v_2 \longrightarrow v_1$ 
in $\widetilde{\Gamma}$, or
\item[\textnormal{(III)}] 
some and all cycles with alternate normal-arrows and dashed-arrows in $\widetilde{\Gamma}$ 
are $\beta$-type.
\end{enumerate}
\end{Cor4.3}

In the case that (II) holds and the cycle $v_1,v_2$ is $\alpha$-type, 
we investigate structures of $\Gamma$ and $\Gamma'$ in Section~\ref{sec5}.
As a preparation, 
we define a finite simple graph $A=A(n;B,C)$.

\begin{Definition}
For a number $n\in {\mathbb{N}}$ at least $3$ 
and subsets $B$ and $C$ of the set $\{1,\ldots,n-1 \}$, 
we define the finite simple graph $A=A(n;B,C)$ as follows:
\begin{enumerate}
\item[(1)] The graph $A$ has $n$ vertices denoted by $V(A)=\{ a_1,a_2,\ldots,a_n \}$.
\item[(2)] $b\in B$ if and only if $[a_{2t-1},a_{2t-1+b}]\in E(A)$ 
for any number $t\in {\mathbb{N}}$ as $1 \le 2t-1 <2t-1+b \le n$.
\item[(3)] $c\in C$ if and only if $[a_{2t},a_{2t+c}]\in E(A)$ 
for any number $t\in {\mathbb{N}}$ as $2 \le 2t <2t+c \le n$.
\end{enumerate}
\end{Definition}

We give some examples of $A(n;B,C)$ in Figure~\ref{figure-A}.

\begin{figure}[h]
\centering
{
\vspace{1mm}
\includegraphics[keepaspectratio, scale=0.90, bb=0 0 396 345]{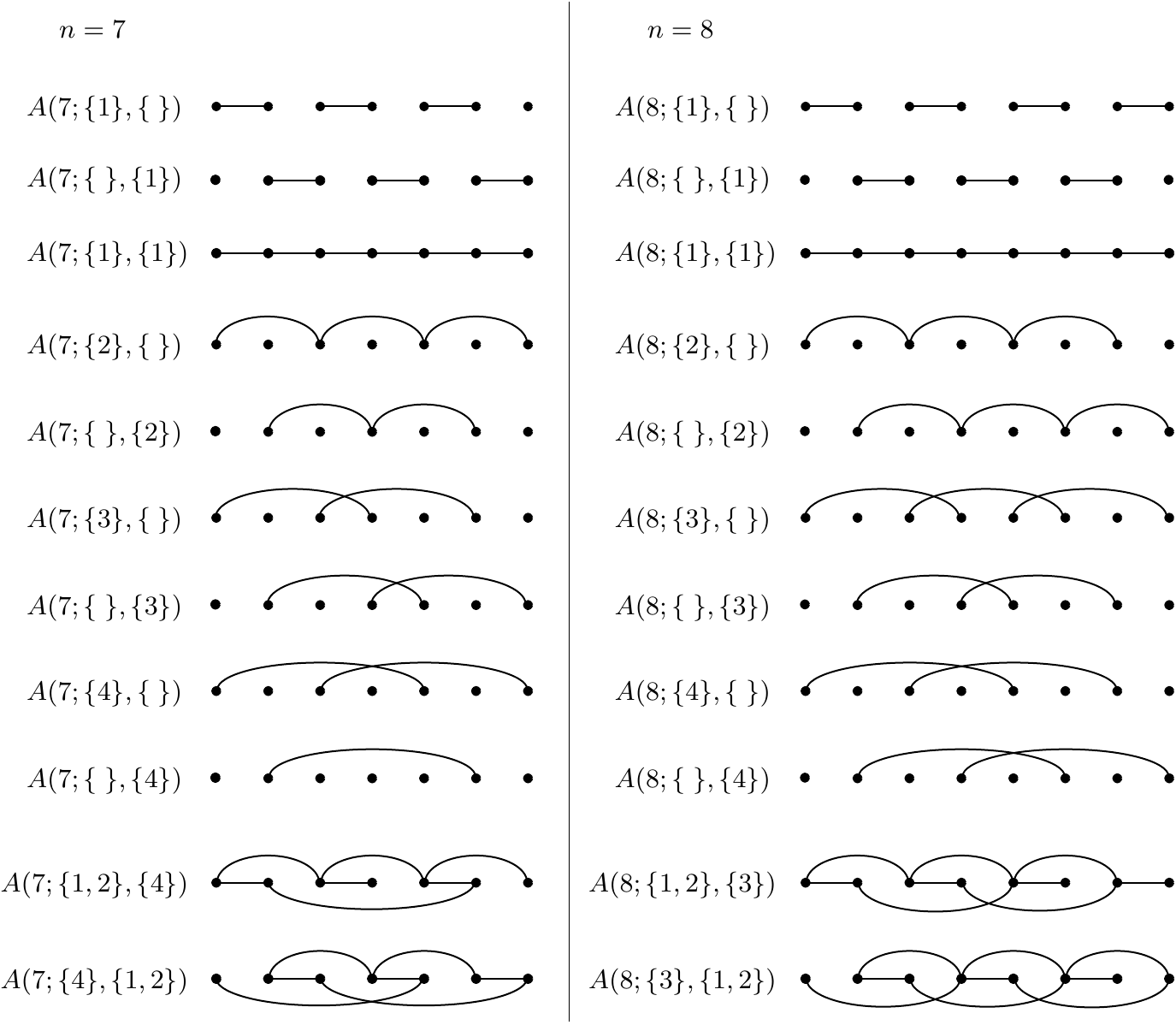}
\vspace*{-2mm}
}
\caption{Examples of graphs $A(n;B,C)$ for $n=7,8$}\label{figure-A}
\end{figure}

Then we have the following by the definition of $A(n;B,C)$.

\begin{Lemma}\label{LemA}
Let $n\in {\mathbb{N}}$ be a number at least $3$ 
and let $B$ and $C$ be subsets of the set $\{1,\ldots,n-1 \}$.
Let $A:=A(n;B,C)$ and $A':=A(n;C,B)$, 
where $V(A)=\{a_1,a_n,\ldots,a_n \}$ and $V(A')=\{b_1,b_2,\ldots,b_n \}$ are 
the numbering of vertices as in the definition of $A=A(n;B,C)$ and $A'=A(n;C,B)$.
Let $a'_i:=b_{n-i+1}$ for $i=1,\ldots,n$.
(Hence $V(A')=\{a'_1,a'_2,\ldots,a'_n \}$.)
Then the two graphs $A$ and $A'$ satisfy the following$:$
\begin{enumerate}
\item[\textnormal{(1)}] There exists 
the isomorphism $\Psi:A-\{a_1,a_2 \} \to A-\{ a_{n-1},a_n \}$ 
such that $\Psi(a_i)=a_{i-2}$ for any $i=3,\ldots,n$.
\item[\textnormal{(2)}] There exists 
the isomorphism $\Psi':A'-\{a'_1,a'_2 \} \to A'-\{ a'_{n-1},a'_n \}$ 
such that $\Psi'(a'_i)=a'_{i-2}$ for any $i=3,\ldots,n$.
\item[\textnormal{(3)}] There exists 
the isomorphism $\phi_1:A-a_1 \to A'-a'_1$ 
such that $\phi_1(a_i)=a'_{n-i+2}$ for any $i=2,\ldots,n$.
\item[\textnormal{(4)}] There exists 
the isomorphism $\phi_2:A-a_n \to A'-a'_n$ 
such that $\phi_2(a_i)=a'_{n-i}$ for any $i=1,\ldots,n-1$.
\end{enumerate}
\end{Lemma}

\begin{Definition}
Let $A:=A(n;B,C)$ where $n$ is a number at least $3$ 
and $B$ and $C$ are subsets of the set $\{1,\ldots,n-1 \}$.
Then we define the numbers $\beta(A)$ and $\gamma(A)$ as follows;
\begin{align*}
&\beta(A):= \#\{ b\in B \,|\, n-b \ \text{is odd} \} \ \text{and} \\
&\gamma(A):=\#\{ c\in C \,|\, n-c \ \text{is odd} \}.
\end{align*}
Here we note that [$n-b$ is odd] if and only if 
[\ [$b$ is even if $n$ is odd] and [$b$ is odd if $n$ is even]\ ].
\end{Definition}

For a vertex $v$ of a graph $\Gamma$, 
we denote $\deg_\Gamma v$ as the degree of the vertex $v$ in the graph $\Gamma$.
If $\Gamma$ and $\Gamma'$ are finite simple graphs with at least three vertices 
satisfying the property~$(*)$, then 
it is well known that $\deg_\Gamma v = \deg_{\Gamma'} f(v)$ for all $v \in V(\Gamma)$.

\medskip

In Sections~\ref{sec5} and \ref{sec6}, 
we give some remarks and examples of $A=A(n;B,C)$ and $A'=A(n;C,B)$.
We obtain the following in Section~\ref{sec5}.

\begin{Lem5.7}
Let $A=A(n;B,C)$, $A'=A(n;C,B)$, 
$V(A)=\{a_1,a_n,\ldots,a_n \}$ and $V(A')=\{a'_1,a'_2,\ldots,a'_n \}$ 
as in Lemma~\ref{LemA}.
Then the following two statements are equivalent$:$
\begin{enumerate}
\item[\textnormal{(i)}] $\beta(A)=\gamma(A)$.
\item[\textnormal{(ii)}] $|E(A)|=|E(A')|$, $\deg_A a_1 = \deg_{A'} a'_1$ and $\deg_A a_n = \deg_{A'} a'_n$.
\end{enumerate}
\end{Lem5.7}

In Theorem~\ref{Thm6}, we show that 
for finite simple graphs $\Gamma$ and $\Gamma'$ with at least three vertices satisfying the property~$(*)$, 
if in the associated directed graph $\widetilde{\Gamma}$, 
\begin{enumerate}
\item[\textnormal{(II)}] for some $v_1,v_2\in V(\widetilde{\Gamma})$, 
there exist normal-arrows $v_1 \longrightarrow v_2$ and $v_2 \longrightarrow v_1$ 
in $\widetilde{\Gamma}$
\end{enumerate}
and if the cycle $v_1,v_2$ is $\alpha$-type, 
then $\Gamma$ and $\Gamma'$ have some structure $(\mathcal{F})$.

An example of $\Gamma$ and $\Gamma'$ with the structure $(\mathcal{F})$ 
is given in Figure~\ref{figure-exampleA}.
In this example, $A=A(6;\{1\},\{1\})$ and $A'=A(6;\{1\},\{1\})$.

\begin{figure}[h]
\centering
{
\vspace*{3mm}
\includegraphics[keepaspectratio, scale=0.90, bb=0 0 329 127]{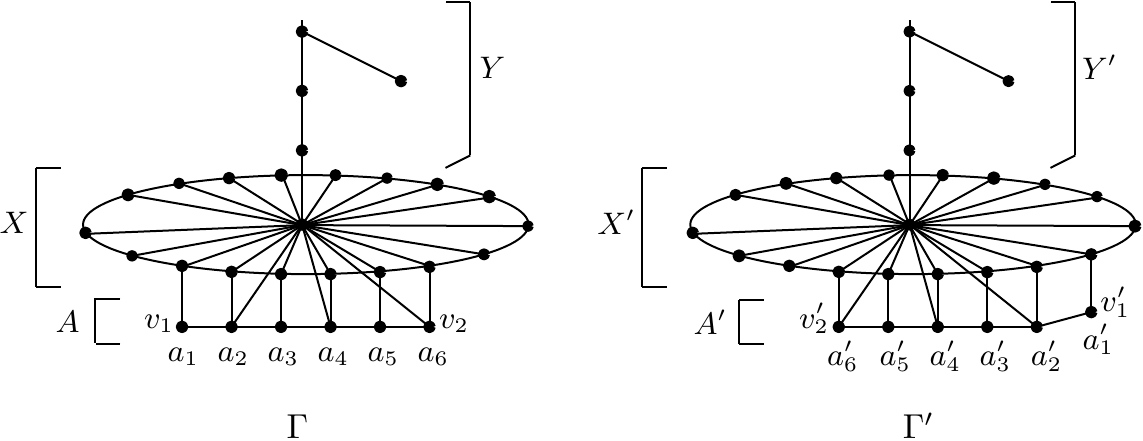}
}
\caption{An example of graphs $\Gamma$ and $\Gamma'$ with the structure $(\mathcal{F})$}\label{figure-exampleA}
\end{figure}

\smallskip

We introduce some examples and remarks on 
$A=A(n;B,C)$ and $A'=A(n;C,B)$ in Section~\ref{sec6} 
and finally we give an example and a remark on $\beta$-type cycles in Section~\ref{sec7}.

\section{Examples on the associated directed graphs}\label{sec2}

We give some examples of the associated directed graphs.

\begin{Example}\label{Example1}
Let $\Gamma$ and $\Gamma'$ be the graphs as Figure~\ref{figure1}~(I-1).
We define the map $f:V(\Gamma) \to V(\Gamma')$ as $f(v_i)=v'_i$ for $i=1,2,3,4,5$.
For each $v_i \in V(\Gamma)$, we define the isomorphism 
$\phi_{v_i}:\Gamma -v_i \to \Gamma' - v'_i$ as Figure~\ref{figure1}~(I-2).
Then the associated directed graph 
$\widetilde{\Gamma}=\widetilde{\Gamma}(\Gamma,\Gamma',f,\{\phi_v\}_{v\in V(\Gamma)})$ 
becomes as Figure~\ref{figure1}~(I-3).

Here we note that 
for each vertex $v \in V(\widetilde{\Gamma})$, 
the normal-arrows and the dashed-arrows 
arising from $v$ in $\widetilde{\Gamma}$ 
are determined by the isomorphism 
$\phi_{v}:\Gamma-v \to \Gamma'-f(v)$.

In this example, 
there are no normal-arrows arising from $v_2$ (and no dashed-arrows arising from $v_2$).
Then there exists the isomorphism $\overline{\phi_{v_2}}:\Gamma \to \Gamma'$ 
as $\overline{\phi_{v_2}}|_{\Gamma - v_2}=\phi_{v_2}$ and $\overline{\phi_{v_2}}(v_2)=f(v_2)=v'_2$ 
(as Theorem~\ref{Thm1}).
\end{Example}

\begin{Example}\label{Example2}
Let $\Gamma$ and $\Gamma'$ be the graphs as Figure~\ref{figure2}~(II-1).
We define the map $f:V(\Gamma) \to V(\Gamma')$ as $f(v_i)=v'_i$ for $i=1,2,3,4$.
For each $v_i \in V(\Gamma)$, we define the isomorphism 
$\phi_{v_i}:\Gamma -v_i \to \Gamma' - v'_i$ as Figure~\ref{figure2}~(II-2).
Then the associated directed graph 
$\widetilde{\Gamma}=\widetilde{\Gamma}(\Gamma,\Gamma',f,\{\phi_v\}_{v\in V(\Gamma)})$ 
becomes as Figure~\ref{figure2}~(II-3).

In this example, 
there are no normal-arrows (and no dashed-arrows) arising from $v_3$ and 
this implies that 
there exists the isomorphism $\overline{\phi_{v_3}}:\Gamma \to \Gamma'$ 
as $\overline{\phi_{v_3}}|_{\Gamma - v_3}=\phi_{v_3}$ and $\overline{\phi_{v_3}}(v_3)=f(v_3)=v'_3$.
\end{Example}

\begin{Example}\label{Example3}
Let $\Gamma$ and $\Gamma'$ be the graphs as Figure~\ref{figure3}~(III-1) 
that are the same graphs in Example~\ref{Example2}.
We define the map $f:V(\Gamma) \to V(\Gamma')$ as $f(v_i)=v'_i$ for $i=1,2,3,4$.
For each $v_i \in V(\Gamma)$, we define the isomorphism 
$\phi_{v_i}:\Gamma -v_i \to \Gamma' - v'_i$ as Figure~\ref{figure3}~(III-2).
Then the associated directed graph 
$\widetilde{\Gamma}=\widetilde{\Gamma}(\Gamma,\Gamma',f,\{\phi_v\}_{v\in V(\Gamma)})$ 
becomes as Figure~\ref{figure3}~(III-3).

In this example, 
for every vertex $v \in V(\widetilde{\Gamma})$, 
there exists a normal-arrow (and a dashed-arrow) 
arising from $v$ in $\widetilde{\Gamma}$.
Also there are arrows $v_1 \longrightarrow v_4$, 
$v_4 \longrightarrow v_1$, 
$v_1 \dashrightarrow v_4$ and $v_4 \dashrightarrow v_1$ in $\widetilde{\Gamma}$.
\end{Example}

\begin{figure}
\centering
{
\includegraphics[keepaspectratio, scale=0.95, bb=0 0 381 563]{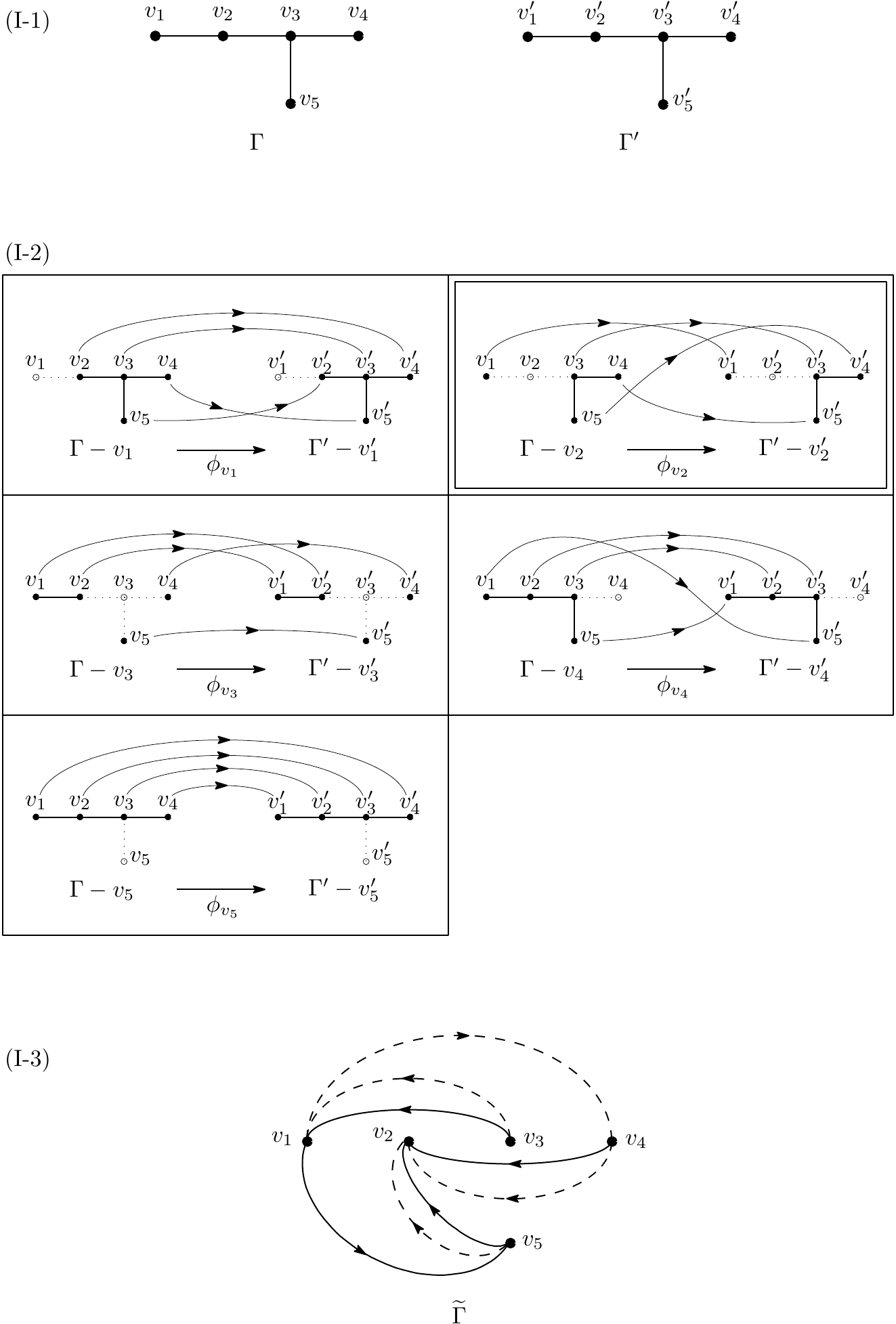}
}
\vspace*{5mm}
\caption{Example~2.1}
\label{figure1}
\vspace*{10mm}
\end{figure}

\begin{figure}
\centering
{
\includegraphics[keepaspectratio, scale=0.95, bb=0 0 360 338]{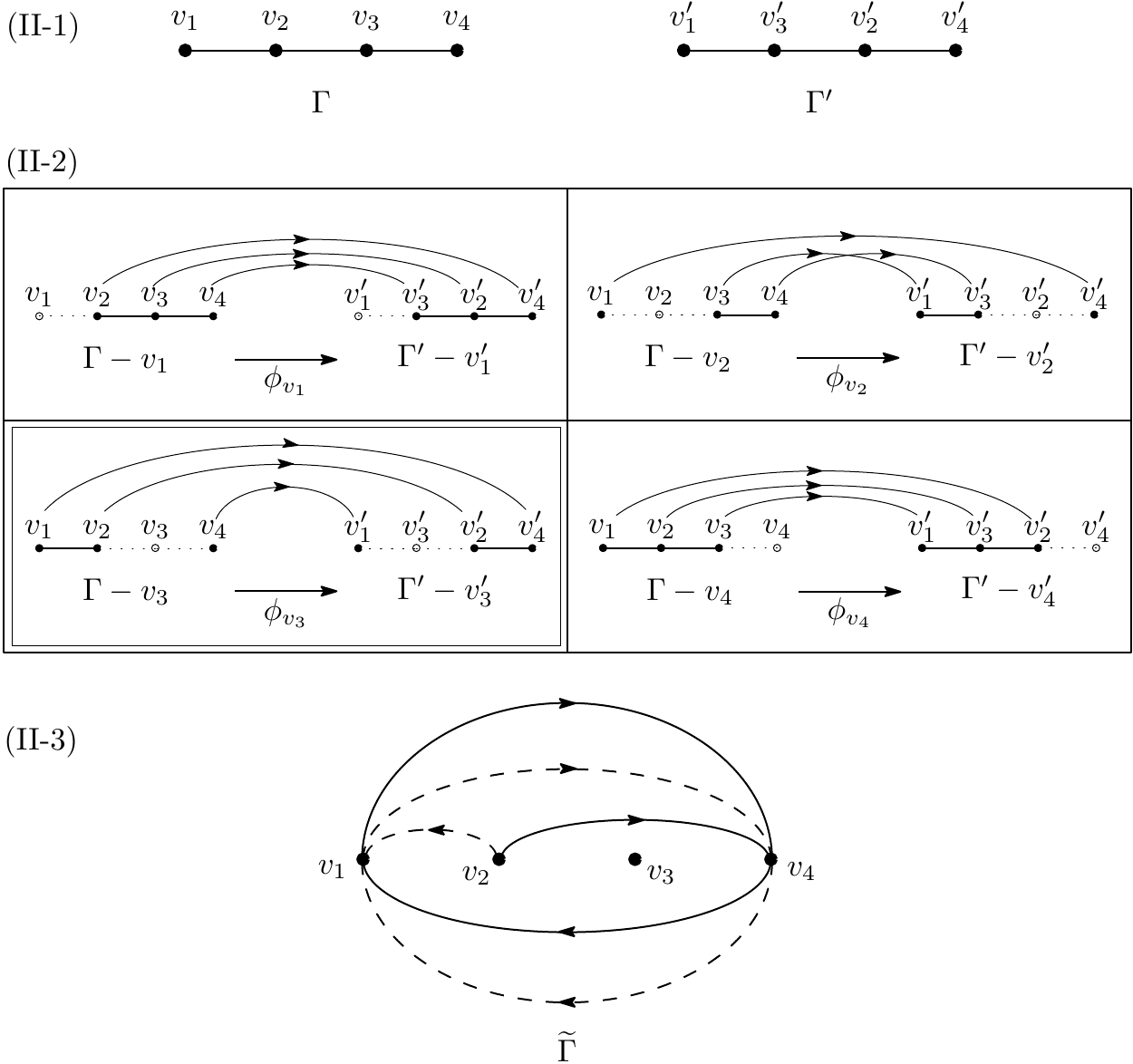}
}
\caption{Example~2.2}
\label{figure2}
\end{figure}

\begin{figure}
\centering
{
\includegraphics[keepaspectratio, scale=0.95, bb=0 0 360 338]{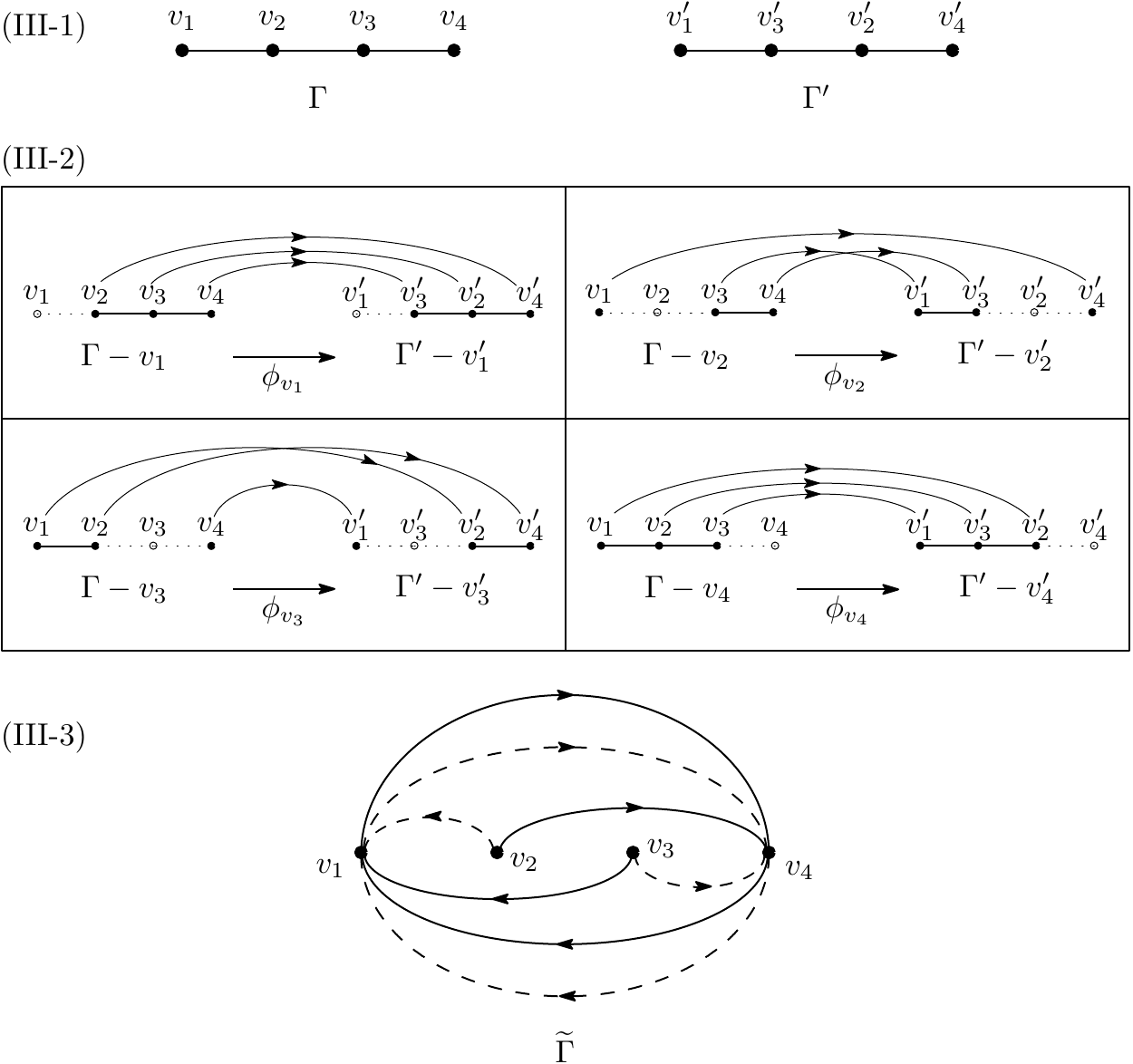}
}
\caption{Example~2.3}
\label{figure3}
\end{figure}

\section{On the associated directed graphs}\label{sec3}

Let $\Gamma$ and $\Gamma'$ be finite simple graphs with at least three vertices 
satisfying the property~$(*)$ and 
let $\widetilde{\Gamma}$ be the associated directed graph.
We investigate some properties of $\widetilde{\Gamma}$.
The following lemma is obtained immediately from the definition of $\widetilde{\Gamma}$.

\begin{Lemma}\label{lemma0}
\begin{enumerate}
\item[\textnormal{(1)}] For $v_1,v_2 \in V(\widetilde{\Gamma})$, 
if there exists a normal-arrow $v_1 \longrightarrow v_2$ in $\widetilde{\Gamma}$ 
then $[v_1,v_2] \not\in E(\Gamma)$.
\item[\textnormal{(2)}] For $v_1,v_2 \in V(\widetilde{\Gamma})$, 
if there exists a dashed-arrow $v_1 \dashrightarrow v_2$ in $\widetilde{\Gamma}$ 
then $[f(v_1),f(v_2)] \not\in E(\Gamma')$.
\end{enumerate}
\end{Lemma}

We show a proposition.

\begin{Proposition}\label{Proposition1}
For all $v\in V(\widetilde{\Gamma})$, 
the number of normal-arrows arising from $v$ and 
the number of dashed-arrows arising from $v$ are equal in $\widetilde{\Gamma}$.
\end{Proposition}

\begin{proof}
For all $v\in V(\Gamma)$, 
$\deg_{\Gamma} v =\deg_{\Gamma'} f(v)$ holds.

Let $v$ be a fixed vertex of $\Gamma$ and let $v':=f(v)\in V(\Gamma')$.
Then the degree $\deg_{\Gamma} v$ is the sum of the number $P_v$ of edges 
$[v,w]$ as $[v',\phi_{v} (w)]\in E(\Gamma')$ and 
the number $Q_v$ of edges $[v,w]$ as $[v',\phi_{v} (w)]\not\in E(\Gamma')$.
Also, the degree $\deg_{\Gamma'} v'$ is the sum of the number $P'_{v'}$ of edges 
$[v',u']$ as $[v,\phi_{v}^{-1} (u')]\in E(\Gamma)$ and 
the number $Q'_{v'}$ of edges $[v',u']$ as $[v,\phi_{v}^{-1}(u')]\not\in E(\Gamma)$.
Hence
\begin{align*}
&P_v = \# \{ w\in V(\Gamma-v) \,|\, [v,w]\in E(\Gamma) \ \text{and}\ [v',\phi_{v}(w)]\in E(\Gamma')\}, \\
&Q_v = \# \{ w\in V(\Gamma-v) \,|\, [v,w] \in E(\Gamma) \ \text{and}\ [v',\phi_{v}(w)] \not\in E(\Gamma')\}, \\
&P'_{v'} = \# \{ u'\in V(\Gamma'-v') \,|\, [v',u']\in E(\Gamma') \ \text{and}\ [v,\phi^{-1}_{v} (u')]\in E(\Gamma)\}, \\
&Q'_{v'} = \# \{ u'\in V(\Gamma'-v') \,|\, [v',u']\in E(\Gamma') \ \text{and}\ [v,\phi_{v}^{-1} (u')]\not\in E(\Gamma)\}, \\
&\deg_{\Gamma} v=P_v + Q_v \ \text{and} \\
&\deg_{\Gamma'} v'=P'_{v'} + Q'_{v'}.
\end{align*}

Let $w\in V(\Gamma-v)$ and let $u':=\phi_{v}(w)$. 
Here $u'\in V(\Gamma'-v')$.
Then $[v',\phi_{v}(w)]\in E(\Gamma')$ if and only if $[v',u']\in E(\Gamma')$.
Also $[v,w]\in E(\Gamma)$ if and only if $[v,\phi^{-1}_{v} (u')]\in E(\Gamma)$.

Hence $P_v =P'_{v'}$ and 
\[ Q_v=\deg_{\Gamma}v-P_v=\deg_{\Gamma'}v'-P'_{v'}= Q'_{v'}. \]
Also
\begin{align*}
Q_v &= \# \{ u'\in V(\Gamma'-v') \,|\, [v,\phi^{-1}_{v} (u')] \in E(\Gamma) \ \text{and}\ [v',u']\not\in E(\Gamma') \} \\
&= \# \{ u\in V(\Gamma-v)\,|\, v \dashrightarrow u \ \text{in}\ \widetilde{\Gamma} \}
\end{align*}
and
\begin{align*}
Q'_{v'} &= \# \{ w\in V(\Gamma-v) \,|\, [v',\phi_{v} (w)]\in E(\Gamma') \ \text{and}\ [v,w]\not\in E(\Gamma) \} \\
&= \# \{ w\in V(\Gamma-v)\,|\, v \longrightarrow w \ \text{in}\  \widetilde{\Gamma} \}.
\end{align*}
Thus, 
the number of normal-arrows arising from $v$ and 
the number of dashed-arrows arising from $v$ are equal in $\widetilde{\Gamma}$.
\end{proof}

We show the following theorem.
This is the first motivation 
for considering the associated directed graph $\widetilde{\Gamma}$.

\begin{Theorem}\label{Thm1}
For any vertex $v \in V(\Gamma)=V(\widetilde{\Gamma})$, 
the following two statements are equivalent$:$
\begin{enumerate}
\item[\textnormal{(i)}] 
There are no normal-arrows arising from $v$ (and no dashed-arrows arising from $v$) in 
the associated directed graph $\widetilde{\Gamma}$.
\item[\textnormal{(ii)}] 
The isomorphism $\phi_{v}:\Gamma-v \to \Gamma'-f(v)$ 
extends to the isomorphism $\overline{\phi_{v}}:\Gamma \to \Gamma'$ 
as $\overline{\phi_{v}}|_{\Gamma-v}=\phi_{v}$ and $\overline{\phi_{v}}(v)=f(v)$.
\end{enumerate}
\end{Theorem}

\begin{proof}
Let $v':=f(v)$.
We suppose that (i) holds.
Then by Proposition~\ref{Proposition1}, 
there are no dashed-arrows arising from $v$, 
$Q_v=Q'_{v'}=0$ and $\deg_{\Gamma} v = P_v = P'_{v'}=\deg_{\Gamma'} v'$.
Hence for $w\in V(\Gamma-v)$, 
$[v,w]\in E(\Gamma)$ if and only if $[v',\phi_{v}(w)]\in E(\Gamma')$.
Thus we can obtain the isomorphism $\overline{\phi_{v}}:\Gamma \to \Gamma'$ 
such that $\overline{\phi_{v}}|_{\Gamma - v}=\phi_{v}$ and $\overline{\phi_{v}}(v)=f(v)=v'$.

We suppose that (ii) holds.
Then for any $w\in V(\Gamma-v)$, 
$[v,w]\in E(\Gamma)$ if and only if $[v',\phi_{v}(w)]\in E(\Gamma)$.
By the definition of $\widetilde{\Gamma}$, 
the statement (i) holds.
\end{proof}

\section{On cycles with alternate normal-arrows and dashed-arrows in $\widetilde{\Gamma}$}\label{sec4}

We suppose that 
there does not exist a vertex $v \in V(\widetilde{\Gamma})$ 
with no normal-arrows arising from $v$ in 
the associated directed graph $\widetilde{\Gamma}$.
Then by the argument in Section~\ref{sec1}, there exist a positive number $k\in {\mathbb{N}}$ and 
a sequence $v_1,\ldots,v_{2k} \in V(\widetilde{\Gamma})$ 
such that there exist arrows $v_{2i-1} \longrightarrow v_{2i}$ and 
$v_{2i} \dashrightarrow v_{2i+1}$ in $\widetilde{\Gamma}$ 
for any $i=1,\ldots,k$, where $v_{2k+1}:=v_1$ 
by Proposition~\ref{Proposition1} (see Figure~\ref{figure-cycle}), 
because $\widetilde{\Gamma}$ is finite.

We define $\alpha$-type and $\beta$-type for such a cycle $v_1,\ldots,v_{2k}$.

\begin{Definition}\label{def:type}
We consider a sequence $\{b'_i\}$ in $V(\Gamma')$ as
\begin{align*}
&b'_1:=v'_1, \\
&b'_2:=\phi_{v_1}(v_2), \\
&b'_3:=\phi_{v_1}\circ \phi^{-1}_{v_2}(v'_3), \\
&b'_4:=\phi_{v_1}\circ \phi^{-1}_{v_2}\circ\phi_{v_3}(v_4), \\
&b'_5:=\phi_{v_1}\circ \phi^{-1}_{v_2}\circ\phi_{v_3}\circ \phi^{-1}_{v_4}(v'_5), \\
&\cdots \\
&b'_{2k}:=\phi_{v_1}\circ \phi^{-1}_{v_2}\circ\phi_{v_3}\circ \phi^{-1}_{v_4}
\circ\cdots \circ \phi_{v_{2k-1}}(v_{2k})=\overline{\phi}_{2k-1}(v_{2k}), \\
&b'_{2k+1}:=\overline{\phi}_{2k-1}\circ \phi^{-1}_{v_{2k}}(v'_1)=\Psi_0(b'_1), \\
&b'_{2k+2}:=\Psi_0(b'_2), \\
&b'_{2k+3}:=\Psi_0(b'_3), \\
&b'_{2k+4}:=\Psi_0(b'_4), \\
&\cdots \\
&b'_{4k}:=\Psi_0(b'_{2k}), \\
&b'_{4k+1}:=(\Psi_0)^2(b'_1), \\
&b'_{4k+2}:=(\Psi_0)^2(b'_2), \\
&b'_{4k+3}:=(\Psi_0)^2(b'_3), \\
&b'_{4k+4}:=(\Psi_0)^2(b'_4), \\
&\cdots,
\end{align*}
where 
$\overline{\phi}_{2k-1}:=\phi_{v_1}\circ \phi^{-1}_{v_2}\circ \phi_{v_3}
\circ \phi^{-1}_{v_4}\circ \cdots \circ \phi_{v_{2k-1}}$ 
and $\Psi_0:=\overline{\phi}_{2k-1} \circ \phi^{-1}_{v_{2k}} 
=\phi_{v_1}\circ \phi^{-1}_{v_2}\circ \phi_{v_3}
\circ \phi^{-1}_{v_4}\circ \cdots \circ \phi_{v_{2k-1}}\circ \phi^{-1}_{v_{2k}}$.

Here if $b'_i$ is not defined, then we consider $b'_i$ is blank 
and we denote $b'_i=\square$.

For example, 
$b'_5=\phi_{v_1}\circ \phi^{-1}_{v_2}\circ \phi_{v_3}\circ \phi^{-1}_{v_4}(v'_5)$ 
is defined if and only if $v'_5 \neq v'_4$, $\phi^{-1}_{v_4}(v'_5)\neq v_3$, $\phi_{v_3}\circ \phi^{-1}_{v_4}(v'_5)\neq v'_2$ 
and $\phi^{-1}_{v_2}\circ \phi_{v_3}\circ \phi^{-1}_{v_4}(v'_5) \neq v_1$.
Here we note that $v'_5 \neq v'_4$ and $\phi^{-1}_{v_4}(v'_5)\neq v_3$ always hold.
Hence in this case, $b'_5$ is blank if and only if 
$\phi_{v_3}\circ \phi^{-1}_{v_4}(v'_5)= v'_2$ or 
$\phi^{-1}_{v_2}\circ \phi_{v_3}\circ \phi^{-1}_{v_4}(v'_5) = v_1$.

Let ${b'}^{(0)}_i:=b'_i$ ($i=1,2,\ldots$) and 
let $A'_0:=\{ {b'}^{(0)}_i \,|\, i=1,2,\ldots \}$.

Similarly, we consider a sequence $\{b_i\}$ in $V(\Gamma)$ as
\begin{align*}
&b_1:=v_2, \\
&b_2:=\phi^{-1}_{v_2}(v'_3), \\
&b_3:=\phi^{-1}_{v_2}\circ\phi_{v_3}(v_4), \\
&b_4:=\phi^{-1}_{v_2}\circ\phi_{v_3}\circ \phi^{-1}_{v_4}(v'_5), \\
&b_5:=\phi^{-1}_{v_2}\circ\phi_{v_3}\circ \phi^{-1}_{v_4}\circ\phi_{v_5}(v_6), \\
&\cdots \\
&b_{2k-1}:=\phi^{-1}_{v_2}\circ\phi_{v_3}\circ \phi^{-1}_{v_4}\circ\cdots \circ \phi_{v_{2k-1}}(v_{2k}), \\
&b_{2k}:=\phi^{-1}_{v_2}\circ\phi_{v_3}\circ \phi^{-1}_{v_4}\circ\cdots \circ \phi_{v_{2k-1}}\circ \phi_{v_{2k}}^{-1}(v'_{1})
=\overline{\phi}_{2k}(v'_{1}), \\
&b_{2k+1}:=\overline{\phi}_{2k}\circ \phi_{v_{1}}(v_2)=\Psi'_0(b_1), \\
&b_{2k+2}:=\Psi'_0(b_2), \\
&b_{2k+3}:=\Psi'_0(b_3), \\
&b_{2k+4}:=\Psi'_0(b_4), \\
&\cdots \\
&b_{4k}:=\Psi'_0(b_{2k}), \\
&b_{4k+1}:=(\Psi'_0)^2(b_1), \\
&b_{4k+2}:=(\Psi'_0)^2(b_2), \\
&b_{4k+3}:=(\Psi'_0)^2(b_3), \\
&b_{4k+4}:=(\Psi'_0)^2(b_4), \\
&\cdots,
\end{align*}
where 
$\overline{\phi}_{2k}:=\phi^{-1}_{v_2}\circ\phi_{v_3}\circ \phi^{-1}_{v_4}\circ\cdots 
\circ \phi_{v_{2k-1}}\circ \phi_{v_{2k}}^{-1}$ 
and $\Psi'_0:=\overline{\phi}_{2k}\circ \phi_{v_{1}} 
=\phi^{-1}_{v_2}\circ\phi_{v_3}\circ \phi^{-1}_{v_4}\circ\cdots 
\circ \phi_{v_{2k-1}}\circ \phi_{v_{2k}}^{-1} \circ \phi_{v_{1}}$.

Let ${b}^{(1)}_i:=b_i$ ($i=1,2,\ldots$) and 
let $A_1:=\{ {b}^{(1)}_i \,|\, i=1,2,\ldots \}$.

Since $v_1,\ldots,v_{2k}$ is a cycle, 
similarly we can define the sequences 
$A'_{2t}=\{ {b'}^{(2t)}_{i} \,|\, i=1,2,\ldots \}$ and 
$A_{2t+1}=\{ b^{(2t+1)}_{i} \,|\, i=1,2,\ldots \}$ for $t=0,1,\ldots,k-1$ 
(where if ${b'}^{(2t)}_{i}$ (and $b^{(2t+1)}_{i}$) is not defined then we consider 
${b'}^{(2t)}_{i}$ (and $b^{(2t+1)}_{i}$) is blank).
Here if $b=\square$ then $\phi_{v_{2t}}(b)=\square$ and $\phi_{v_{2t+1}}^{-1}(b)=\square$.

We say that the cycle $v_1,\ldots,v_{2k}$ is \textit{$\alpha$-type}, 
if there exists a positive number $n_0$ such that for any $t=0,1,\ldots,k-1$, 
\begin{enumerate}
\item[$\cdot$] ${b'}^{(2t)}_{i} \neq \square$ and ${b}^{(2t+1)}_{i} \neq \square$ for all $1 \le i \le n_0$,
\item[$\cdot$] ${b'}^{(2t)}_{i} = \square$ and ${b}^{(2t+1)}_{i} = \square$ for all $i \ge n_0+1$ and 
\item[$\cdot$] ${b'}^{(2t)}_{n_0} = v'_{2t}$ and ${b}^{(2t+1)}_{n_0} = v_{2t+1}$,
\end{enumerate}
where $v'_0=v'_{2k}$ and $A'_0=A'_{2k}$.
In the otherwise case, if there is no such a positive number $n_0$, 
then the cycle $v_1,\ldots,v_{2k}$ is said to be \textit{$\beta$-type}.

If the cycle $v_1,\ldots,v_{2k}$ is \textit{$\alpha$-type} 
then for all $t=1,\ldots,k$, 
$A'_{2t}=\phi_{v_{2t}}(A_{2t-1}-\{v_{2t}\}) \cup \{ v'_{2t} \}$ 
and $A_{2t+1} = \phi_{v_{2t+1}}^{-1} (A'_{2t}-\{v'_{2t+1}\}) \cup \{ v_{2t+1} \}$ hold (see Figure~\ref{figureZ1}).

\smallskip

We also consider a cycle $v_1,v_2 \in V(\widetilde{\Gamma})$ such that 
normal-arrows $v_1 \longrightarrow v_2$ and $v_2 \longrightarrow v_1$ are in $\widetilde{\Gamma}$.
Then similarly we define \textit{$\alpha$-type} and \textit{$\beta$-type} for the cycle $v_1,v_2$ 
by the above corresponding sequences $\{ b'_i \}$ and $\{ b_i \}$.
\end{Definition}

\begin{figure}
\centering
{
\includegraphics[keepaspectratio, scale=0.90, bb=0 0 397 277]{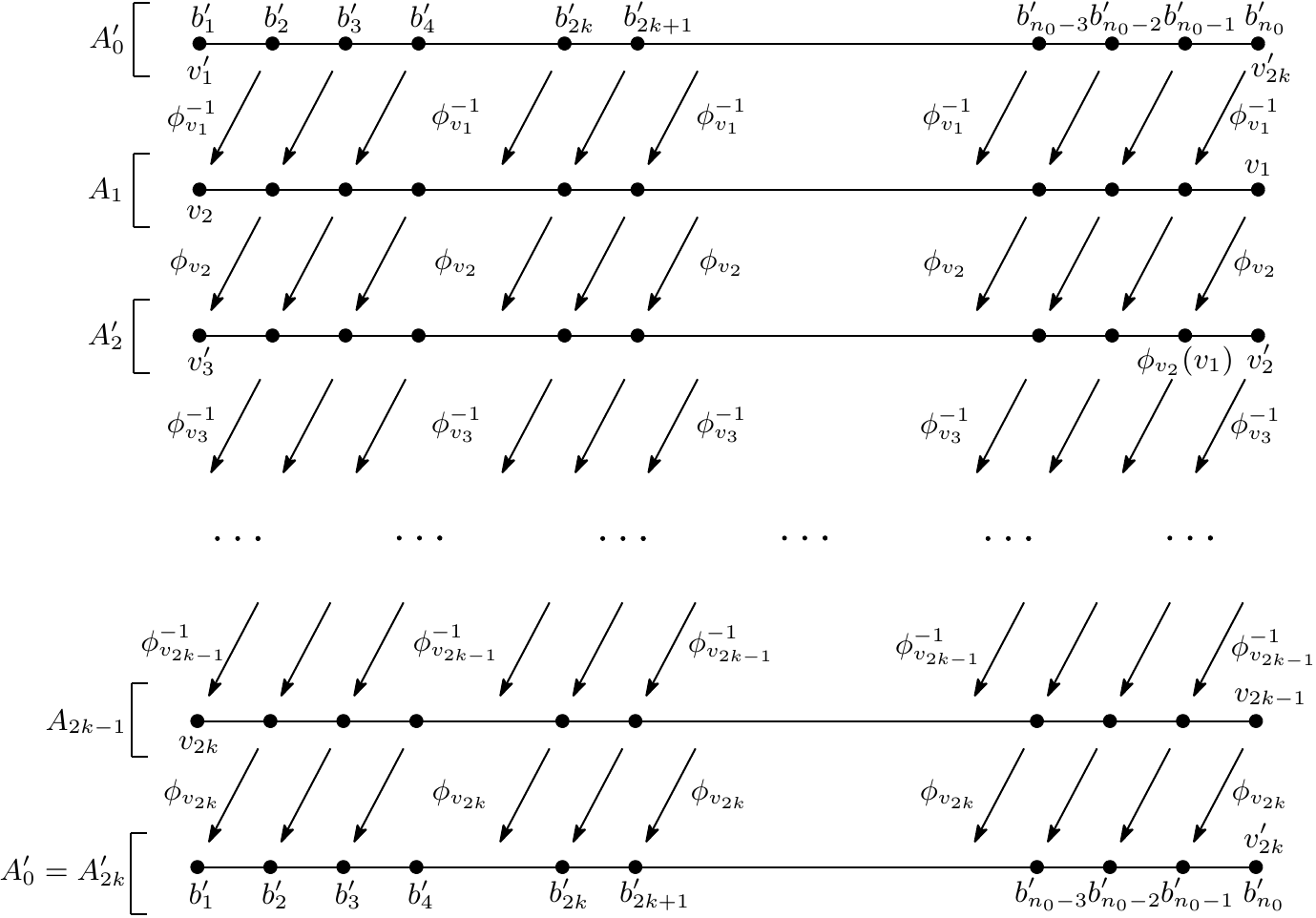}
\vspace*{-2mm}
}
\caption{$A'_{2i}$ and $A_{2i+1}$}
\label{figureZ1}
\end{figure}

We obtain the following propostion.

\begin{Proposition}\label{Prop3}
If for some positive number $k\in {\mathbb{N}}$ and 
some sequence $v_1,\ldots,v_{2k} \in V(\widetilde{\Gamma})$, 
there exist arrows $v_{2i-1} \longrightarrow v_{2i}$ and 
$v_{2i} \dashrightarrow v_{2i+1}$ in $\widetilde{\Gamma}$ 
for any $i=1,\ldots,k$ (where $v_{2k+1}:=v_1$) 
and if the cycle $v_1,\ldots,v_{2k}$ is $\alpha$-type, 
then there exist normal-arrows $v_1 \longrightarrow v_2$ and $v_2 \longrightarrow v_1$ 
in $\widetilde{\Gamma}$.
\end{Proposition}

\begin{proof}
We consider the sequence $\{b'_i\}$ in $V(\Gamma')$ as in Definition~\ref{def:type}.
Here the cycle $v_1,\ldots,v_{2k}$ is $\alpha$-type.
Let $n_0$ be the positive number such that $b'_i \neq \square$ for $1\le i \le n_0$ and 
$b'_i = \square$ for $i \ge n_0+1$.
Let $A'_0:=\{b'_i\,|\,i=1,\ldots,n_0\}$.
Here $b'_i\neq b'_j$ if $i\neq j$.

Then there exist the following edges in $\Gamma'$.

(1) Since $v_1 \longrightarrow v_2$ is in $\widetilde{\Gamma}$, 
we have $[v'_1,\phi_{v_1}(v_2)]=[b'_1,b'_2] \in E(\Gamma')$.

(2) Since $v_2 \dashrightarrow v_3$ is in $\widetilde{\Gamma}$, 
we have $[v_2,\phi^{-1}_{v_2}(v'_3)]\in E(\Gamma)$ and 
$[\phi_{v_1}(v_2),\phi_{v_1}\circ \phi^{-1}_{v_2}(v'_3)]=[b'_2,b'_3]\in E(\Gamma')$.
(Here $\phi^{-1}_{v_2}(v'_3) \neq v_1$. Indeed, $[v_2,\phi^{-1}_{v_2}(v'_3)]\in E(\Gamma)$ and 
$[v_1,v_2] \not\in E(\Gamma)$, since $v_1 \longrightarrow v_2$ is in $\widetilde{\Gamma}$.)

(3) Since $v_3 \longrightarrow v_4$ is in $\widetilde{\Gamma}$, 
we have $[v'_3,\phi_{v_3}(v_4)]\in E(\Gamma')$ and 
$[\phi_{v_1}\circ \phi^{-1}_{v_2}(v'_3), 
\phi_{v_1}\circ \phi^{-1}_{v_2}\circ \phi_{v_3}(v_4)]=[b'_3,b'_4]\in E(\Gamma')$.

(4) Since $v_4 \dashrightarrow v_5$ is in $\widetilde{\Gamma}$, 
we have $[v_4,\phi^{-1}_{v_4}(v'_5)]\in E(\Gamma)$ and 
$[\phi_{v_1}\circ \phi^{-1}_{v_2}\circ \phi_{v_3}(v_4),
\phi_{v_1}\circ \phi^{-1}_{v_2}\circ \phi_{v_3}\circ \phi^{-1}_{v_4}(v'_5)]=[b'_4,b'_5]\in E(\Gamma')$.

$(2k-1)$ Since $v_{2k-1} \longrightarrow v_{2k}$ is in $\widetilde{\Gamma}$, 
we have $[v'_{2k-1},\phi_{v_{2k-1}}(v_{2k})]\in E(\Gamma')$ and 
for 
\[ \overline{\phi}_{2k-2}:=\phi_{v_1}\circ \phi^{-1}_{v_2}\circ \phi_{v_3}
\circ \phi^{-1}_{v_4}\circ \cdots \circ\phi_{v_{2k-3}} \circ \phi^{-1}_{v_{2k-2}}, \]
$[\overline{\phi}_{2k-2}(v'_{2k-1}), \overline{\phi}_{2k-2}\circ \phi_{v_{2k-1}}(v_{2k})]=[b'_{2k-1},b'_{2k}]\in E(\Gamma')$.

$(2k)$ Since $v_{2k} \dashrightarrow v_1$ is in $\widetilde{\Gamma}$, 
we have $[v_{2k},\phi^{-1}_{v_{2k}}(v'_1)]\in E(\Gamma)$ and 
for 
\[ \overline{\phi}_{2k-1}:=\phi_{v_1}\circ \phi^{-1}_{v_2}\circ \phi_{v_3}
\circ \phi^{-1}_{v_4}\circ \cdots \circ\phi_{v_{2k-3}} \circ \phi^{-1}_{v_{2k-2}}\circ \phi_{v_{2k-1}}, \]
$[ \overline{\phi}_{2k-1}(v_{2k}),\overline{\phi}_{2k-1}\circ \phi^{-1}_{v_{2k}}(v'_1)] =[b'_{2k},b'_{2k+1}]\in E(\Gamma')$.

Thus $[b'_i,b'_{i+1}]\in E(\Gamma')$ for any $i=1,\ldots,n_0-1$.

Let $A_{2i-1}:=\phi^{-1}_{v_{2i-1}}(A'_{2i-2}-\{ v'_{2i-1} \})\cup\{ v_{2i-1} \}$ and 
let $A'_{2i}:=\phi_{v_{2i}}(A_{2i-1}-\{ v_{2i} \})\cup\{ v'_{2i} \}$ for each $i=1,\ldots,k$.
Here $A'_{2k}=A'_0$.
We can also denote 
$A_{2i-1}=\{ b^{(2i-1)}_{1},\ldots, b^{(2i-1)}_{n_0} \}$ and 
$A'_{2i}=\{ {b'}^{(2i)}_{1},\ldots, {b'}^{(2i)}_{n_0} \}$ for $i=1,\ldots,k$ 
as in Definition~\ref{def:type} (see Figure~\ref{figureZ1}).

By the definition of $\alpha$-type, $b_{n_0}^{(1)} = v_1$ on $A_1$ 
and ${b'}_{n_0}^{(2)} = v'_2$ on $A'_2$.
Hence $[\phi_{v_2}(v_1),v'_2]\in E(\Gamma')$ on $A'_2$.
Also $[v_1,v_2]\not\in E(\Gamma)$.
Thus there exists a normal-arrow $v_2 \longrightarrow v_1$ in $\widetilde{\Gamma}$.
(Here by the same argument, we can obtain that 
there exist arrows $v_{2i} \longrightarrow v_{2i-1}$ and 
$v_{2i+1} \dashrightarrow v_{2i}$ in $\widetilde{\Gamma}$ 
for all $i=1,\ldots,k$ where $v_{2k+1}=v_1$.)
\end{proof}

From Proposition~\ref{Prop3}, we obtain the following.

\begin{Corollary}\label{Thm4}
Let $\Gamma$ and $\Gamma'$ be finite simple graphs with at least three vertices 
satisfying the property~$(*)$.
Then for the associated directed graph $\widetilde{\Gamma}$, 
\begin{enumerate}
\item[\textnormal{(I)}] there exists a vertex $v\in V(\widetilde{\Gamma})$ 
with no normal-arrow arising from $v$ in $\widetilde{\Gamma}$ 
(then $\Gamma \cong \Gamma'$ by Theorem~\ref{Thm1}), 
\item[\textnormal{(II)}] for some $v_1,v_2\in V(\widetilde{\Gamma})$,
there exist normal-arrows $v_1 \longrightarrow v_2$ and $v_2 \longrightarrow v_1$ 
in $\widetilde{\Gamma}$, or
\item[\textnormal{(III)}] 
some and all cycles with alternate normal-arrows and dashed-arrows in $\widetilde{\Gamma}$ 
are $\beta$-type.
\end{enumerate}
\end{Corollary}

\section{In the case that 
$v_1 \longrightarrow v_2$ and $v_2 \longrightarrow v_1$ are in $\widetilde{\Gamma}$ whose cycle is $\alpha$-type}\label{sec5}

We suppose that there exist normal-arrows $v_1 \longrightarrow v_2$ and 
$v_2 \longrightarrow v_1$ in $\widetilde{\Gamma}$ 
for some vertices $v_1,v_2 \in V(\widetilde{\Gamma})$ and 
suppose that the cycle $v_1,v_2$ is $\alpha$-type.

We consider the isomorphism 
\[ \Psi:= \phi_{v_2}^{-1} \circ \phi_{v_1}: \Gamma-v_1-\phi_{v_1}^{-1}(v'_2) 
\xrightarrow[\cong ]{\; \phi_{v_1}| \;}
\Gamma'-v'_1-v'_2
\xrightarrow[\cong ]{\; \phi_{v_2}^{-1}| \;}
\Gamma - v_2 - \phi_{v_2}^{-1}(v'_1).
\]
Let $A$, $X$ and $Y$ be the full-subgraphs of $\Gamma$ induced by 
\begin{align*}
&V(A)=\{a\in V(\Gamma) \;|\; \Psi^i(a)\in \{v_1,\phi_{v_1}^{-1}(v'_2) \} 
\ \text{for some } i\ge 0 \}, \\
&V(X) \! = \! \left\{x\in V(\Gamma) \! - \! \{ v_1,\phi_{v_1}^{-1}(v'_2)\} 
\! \biggm| \! \! \!
\begin{array}{l}
\Psi(x)\neq x \ \text{and} \\
\Psi^i(x) \! \not\in \! \{v_1,\phi_{v_1}^{-1}(v'_2)\} \; \text{for any}\; i\ge 0
\end{array}
\! \! \right\} \ \text{and} \\
&V(Y)=\{y\in V(\Gamma)-\{ v_1,\phi_{v_1}^{-1}(v'_2)\} \;|\; \Psi(y)= y \}.
\end{align*}
Here
\[ V(\Gamma) = V(X)\cup V(Y) \cup V(A) \]
that is a disjoint union.

Let $X\cup Y$ be the full-subgraph of $\Gamma$ with $V(X\cup Y)=V(X)\cup V(Y)$.
Then $\Psi:X\cup Y \to X\cup Y$ is a graph-automorphism.
Here $X$ is the moving-part and $Y$ is the fixed-part by $\Psi$.
If $X$ is non-empty then 
the graph-automorphism group $G$ of $X\cup Y$ generated by $\Psi$ is 
\[ G=\{\Psi^i \,|\, i\in {\mathbb{Z}} \} \cong {\mathbb{Z}}_m  \]
for some number $m \ge 2$, since $\Gamma$ is a finite graph.

Also we consider the isomorphism 
\[ \Psi':= \phi_{v_2} \circ \phi_{v_1}^{-1}: 
\Gamma'-v'_1-\phi_{v_1}(v_2) 
\xrightarrow[\cong ]{\; \phi_{v_1}^{-1}| \;}
\Gamma-v_1-v_2
\xrightarrow[\cong ]{\; \phi_{v_2}| \;}
\Gamma'-v'_2-\phi_{v_2}(v_1).
\]
Let $A'$, $X'$ and $Y'$ be the full-subgraphs of $\Gamma'$ induced by 
\begin{align*}
&V(A')=\{a'\in V(\Gamma') \;|\; (\Psi')^i(a')\in \{ v'_1,\phi_{v_1}(v_2)\} 
\ \text{for some } i\ge 0 \}, \\
&V(X') \! = \! \left\{ x'\in V(\Gamma') \! - \! \{ v'_1,\phi_{v_1}(v_2)\} 
\! \biggm| \! \! \! 
\begin{array}{l}
\Psi'(x')\neq x' \ \text{and} \\
(\Psi')^i(x') \! \not\in \! \{ v'_1,\phi_{v_1}(v_2)\} \; \text{for any}\; i\ge 0
\end{array}
\! \! \right\} \ \text{and} \\
&V(Y')=\{y'\in V(\Gamma')-\{v'_1,\phi_{v_1}(v_2)\} \;|\; \Psi'(y')= y' \}.
\end{align*}
Here
\[ V(\Gamma') = V(X')\cup V(Y') \cup V(A') \]
that is a disjoint union.

\medskip

We consider a sequence $\{a_i\}$ in $V(\Gamma)$ as 
\begin{align*}
&a_1:=v_1, \\
&a_2:=\phi^{-1}_{v_1}(v'_2), \\
&a_3:=\phi^{-1}_{v_1}\circ \phi_{v_2}(v_1), \\
&a_4:=\phi^{-1}_{v_1}\circ \phi_{v_2}\circ\phi^{-1}_{v_1}(v'_2), \\
&a_5:=\phi^{-1}_{v_1}\circ \phi_{v_2}\circ\phi^{-1}_{v_1}\circ \phi_{v_2}(v_1), \\
&\cdots.
\end{align*}
Here the cycle $v_1,v_2$ is $\alpha$-type by assumption.
Let $n$ be the positive number 
such that $a_i \neq \square$ for all $1 \le i \le n$ and 
$a_i = \square$ for all $i \ge n+1$.
Let $A_0:=\{a_1,a_2,\ldots, a_n \}$.
Here $a_i\neq a_j$ if $i\neq j$.

We also define a sequence $\{a'_i\,|\,i=1,\ldots,n\}$ in $V(\Gamma')$ as 
$a'_i:=\phi_{v_1}(a_{n-i+2})$ for any $i=2,\ldots,n$ and 
$a'_1:=v'_1$.
Let $A'_0:= \{a'_1,a'_2,\ldots, a'_n\}$.
Here $A'_0=\phi_{v_1}(A_0-\{ v_1 \})\cup \{ v'_1 \}$.

Then 
by the argument in Section~\ref{sec4} (where $k=1$ and $n_0=n$), 
we obtain that 
\begin{enumerate}
\item[(1)] $V(A)=A_0$, 
\item[(2)] $V(A')=A'_0$, 
\item[(3)] $\phi_{v_1}(A-a_1)=A'-a'_1$ where $a_1=v_1$ and $a'_1=v'_1$,
\item[(4)] $\phi_{v_1}(a_i)=a'_{n-i+2}$ for any $i=2,\ldots,n$, 
\item[(5)] $\phi_{v_2}(A-a_n)=A'-a'_n$ where $a_n=v_2$ and $a'_n=v'_2$, and
\item[(6)] $\phi_{v_2}(a_i)=a'_{n-i}$ for any $i=1,\ldots,n-1$ 
\end{enumerate}
(see Figure~\ref{figure-D2}).
Here each vertices $a_i, a_{i+1}$ do not have to span an edge in $\Gamma$.

Then $A'_0=\phi_{v_1}(A_0-\{ v_1 \})\cup \{ v'_1 \}$ and 
$A'_0=\phi_{v_2}(A_0-\{ v_2 \})\cup \{ v'_2 \}$ both hold.
Here
\begin{align*}
&a_1=v_1,\ \ a_2=\phi_{v_1}^{-1}(v'_2),\ \ a_{n-1}=\phi_{v_2}^{-1}(v'_1),\ \ a_n=v_2, \\
&a'_1=v'_1,\ \ a'_2=\phi_{v_1}(v_2),\ \ a'_{n-1}=\phi_{v_2}(v_1)\ \ \text{and}\ \ a'_n=v'_2.
\end{align*}

\begin{figure}
\centering
{
\includegraphics[keepaspectratio, scale=0.90, bb=0 0 329 352]{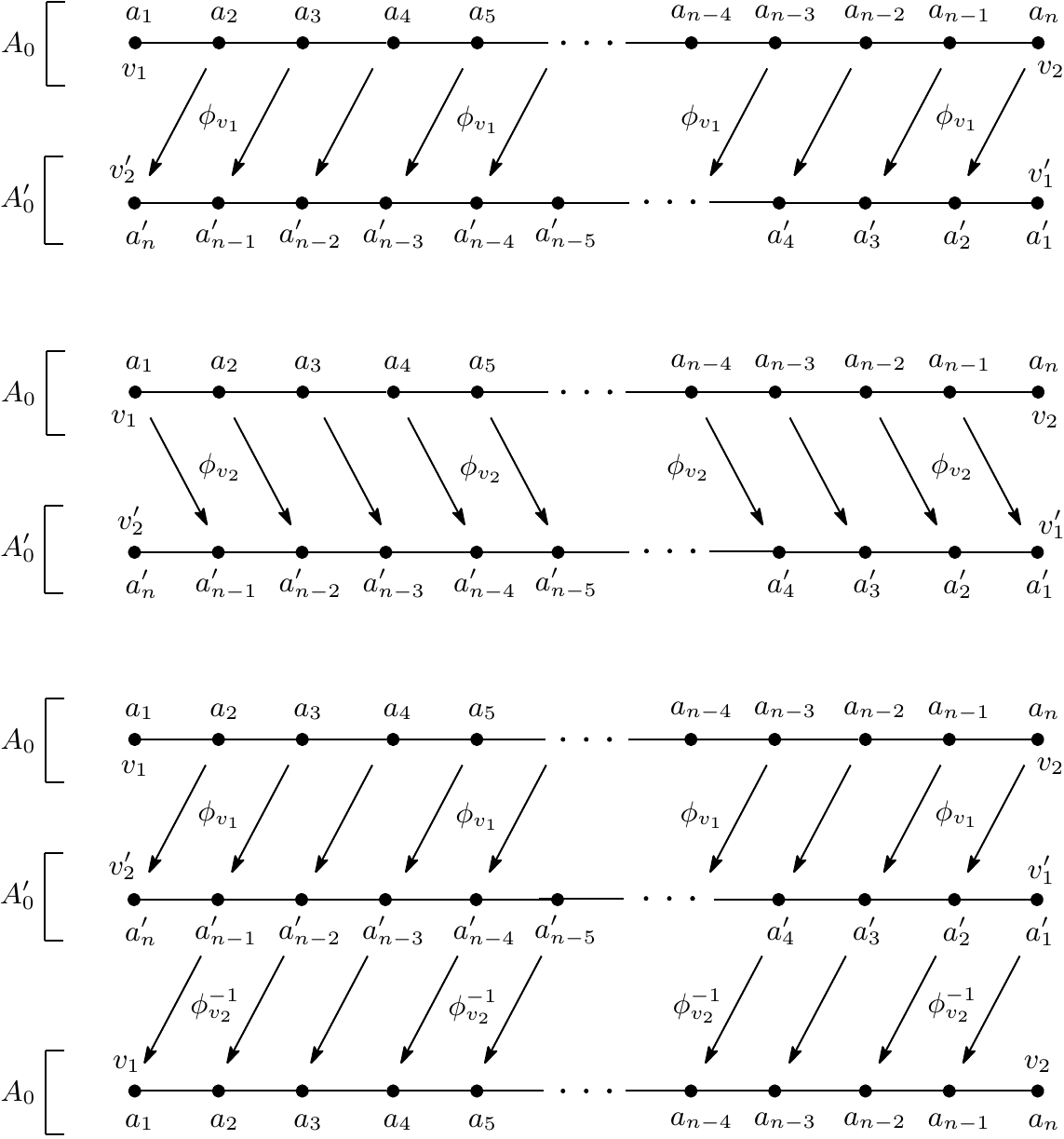}
}
\caption{$A_0$ and $A'_0$}\label{figure-D2}
\end{figure}

Also
\begin{enumerate}
\item[(7)] $\Psi= \phi_{v_2}^{-1} \circ \phi_{v_1}: \Gamma-\{a_1,a_2\} \to \Gamma-\{a_{n-1},a_n\}$ 
is an isomorphism such that $\Psi(a_i)=a_{i-2}$ for any $i=3,\ldots,n$,
\item[(8)] $\Psi'= \phi_{v_2} \circ \phi_{v_1}^{-1}: \Gamma'-\{a'_1,a'_2\} \to \Gamma'-\{a'_{n-1},a'_n\}$ 
is an isomorphism such that $\Psi'(a'_i)=a'_{i-2}$ for any $i=3,\ldots,n$,
\item[(9)] $\Psi|_{A-\{a_1,a_2\}}: A-\{a_1,a_2\} \to A-\{a_{n-1},a_n\}$ 
is an isomorphism such that $\Psi(a_i)=a_{i-2}$ for any $i=3,\ldots,n$, and
\item[(10)] $\Psi'|_{A'-\{a'_1,a'_2\}}: A'-\{a'_1,a'_2\} \to A'-\{a'_{n-1},a'_n\}$ 
is an isomorphism such that $\Psi'(a'_i)=a'_{i-2}$ for any $i=3,\ldots,n$.
\end{enumerate}

We obtain the following.

\begin{Proposition}
Then $\phi_{v_i}(X\cup Y)=X'\cup Y'$, 
$\phi_{v_i}(X)=X'$ and $\phi_{v_i}(Y)=Y'$ 
for each $i=1,2$.
Hence 
\[ X\cup Y \cong X'\cup Y',\ X\cong X' \ \text{and}\ Y \cong Y'. \]
Their isomorphisms are obtained by the restrictions of $\phi_{v_1}$ and $\phi_{v_2}$ both.
\end{Proposition}

\begin{proof}
Since $\phi_{v_1}(A-v_1)=A'-v'_1$ and $\phi_{v_2}(A-v_2)=A'-v'_2$ by (3) and (5) above, 
we have that $\phi_{v_i}(X\cup Y)=X'\cup Y'$ for each $i=1,2$.

We show that $\phi_{v_1}(X)=X'$.
Let $x_0 \in V(X)$.
We consider $x'_1:=\phi_{v_1}(x_0)$, 
$x_2:=\phi_{v_2}^{-1}(x'_1)=\Psi(x_0)$ 
and $x'_3:=\phi_{v_1}(x_2)=(\Psi')^{-1}(x'_1)$.
Then $x_0 \neq x_2$, since $x_0 \in V(X)$.
Hence $x'_1=\phi_{v_1}(x_0) \neq \phi_{v_1}(x_2)=x'_3$.
Thus $\Psi'(x'_1)\neq x'_1$ and $x'_1 \in V(X')$.
This follows that $\phi_{v_1}(x_0) \in V(X')$ for all $x_0 \in V(X)$.
Thus $X\cong X'$ by $\phi_{v_1}$.

Also since $V(Y)=V(X\cup Y)-V(X)$ and $V(Y')=V(X'\cup Y')-V(X')$, 
we obtain that $\phi_{v_1}(y) \in V(Y')$ for all $y \in V(Y)$.
Thus $Y\cong Y'$ by $\phi_{v_1}$.

By the same argument, we can also show that 
$X\cong X'$ and $Y\cong Y'$ by $\phi_{v_2}$.
\end{proof}

\medskip

By (1)--(10) above, we can denote $A=A(n;B,C)$ and $A'=A(n;C,B)$ 
for some subsets $B$ and $C$ of the set $\{1,\ldots,n-1 \}$.
Here $V(A)=\{a_1,a_n,\ldots,a_n \}$ is the numbering of vertices 
in the definition of $A=A(n;B,C)$.
Also if $V(A')=\{b_1,b_2,\ldots,b_n \}$ is 
the numbering of vertices in the definition of $A'=A(n;C,B)$, then 
$a'_i=b_{n-i+1}$ for all $i=1,\ldots,n$ and $V(A')=\{a'_1,a'_2,\ldots,a'_n \}$.

\medskip

Then 
\begin{enumerate}
\item[(a)] since $v_1 \longrightarrow v_2$ is in $\widetilde{\Gamma}$,
we have that $[a'_1,a'_2]=[v'_1,\phi_{v_1}(v_2)]\in E(\Gamma')$ holds, and
\item[(b)] since $v_2 \longrightarrow v_1$ is in $\widetilde{\Gamma}$,
we have that $[a'_n,a'_{n-1}]=[v'_2,\phi_{v_2}(v_1)]\in E(\Gamma')$ holds.
\end{enumerate}
Here $1\in C$, because $A'=A(n;C,B)$ and $[a'_n,a'_{n-1}]=[b_1,b_2] \in E(A')$.

Also 
\begin{enumerate}
\item[(c)] $v_1 \dashrightarrow v_2$ is in $\widetilde{\Gamma}$ if and only if 
$[a_1,a_2]=[v_1,\phi_{v_1}^{-1}(v'_2)]\in E(\Gamma)$, and
\item[(d)] $v_2 \dashrightarrow v_1$ is in $\widetilde{\Gamma}$ if and only if 
$[a_n,a_{n-1}]=[v_2,\phi_{v_2}^{-1}(v'_1)]\in E(\Gamma)$.
\end{enumerate}

Thus we obtain the following.

\begin{Lemma}\label{lemma5.2}
We can denote $A=A(n;B,C)$ and $A'=A(n;C,B)$ 
for some subsets $B$ and $C$ of the set $\{1,\ldots,n-1 \}$.
Here $1\in C$.
\end{Lemma}

Then we show some lemmas.

\begin{Lemma}\label{lemmaA}
The following statements are equivalent$:$
\begin{enumerate}
\item[\textnormal{(1)}] The sequence $a_1,\ldots,a_n$ is a Hamilton-path of $A$.
\item[\textnormal{(2)}] The sequence $a'_1,\ldots,a'_n$ is a Hamilton-path of $A'$.
\item[\textnormal{(3)}] There exist dashed-arrows 
$v_1 \dashrightarrow v_2$ and $v_2 \dashrightarrow v_1$ in $\widetilde{\Gamma}$.
\end{enumerate}
\end{Lemma}

\begin{proof}
(1)$\Longleftrightarrow$(2): 
For $A=A(n;B,C)$ and $A'=A(n;C,B)$,
the sequence $a_1,\ldots,a_n$ is a Hamilton-path of $A$ 
if and only if $[a_1,a_2]$ and $[a_2,a_3]$ are in $E(A)$; 
that is, $1\in B$ and $1 \in C$.
This is equivalent that 
the sequence $b_1,\ldots,b_n$ is a Hamilton-path of $A'=A(n;C,B)$ 
where $a'_i=b_{n-i+1}$ for $i=1,\ldots,n$.
Hence (1) and (2) are equivalent.

[(1) and (2)]$\Longleftrightarrow$(3): 
Now a dashed-arrow $v_1 \dashrightarrow v_2$ is in $\widetilde{\Gamma}$ 
if and only if $[a_1,a_2] \in E(A)$.
Also a dashed-arrow $v_2 \dashrightarrow v_1$ is in $\widetilde{\Gamma}$ 
if and only if $[a_n,a_{n-1}] \in E(A)$.

Thus, if (1) and (2) hold then we obtain (3) holds.

Suppose that (3) holds.
Then $[a_1,a_2] \in E(A)$ by the above.
Hence $1\in B$ in $A=A(n;B,C)$.
Also $1\in C$ by Lemma~\ref{lemma5.2}.
Thus (1) and (2) hold.
\end{proof}

\begin{Lemma}\label{LemmaC}
If $n$ is odd, then 
the sequence $a_1,\ldots,a_n$ is a Hamilton-path of $A$, 
the sequence $a'_1,\ldots,a'_n$ is a Hamilton-path of $A'$ and 
there exist dashed-arrows 
$v_1 \dashrightarrow v_2$ and $v_2 \dashrightarrow v_1$ in $\widetilde{\Gamma}$.
\end{Lemma}

\begin{proof}
Suppose that $n$ is odd.
By the assumption, 
a normal-arrow $v_1 \longrightarrow v_2$ is in $\widetilde{\Gamma}$ and 
$[a'_1,a'_2]=[b_n,b_{n-1}] \in E(A')$.
Then since $n$ is odd, $[a'_{n-1},a'_{n-2}]=[b_2,b_3] \in E(A')$.
Hence $1\in B$, because $A'=A(n;C,B)$.
Thus, $[a_1,a_2]\in E(A)$.
Also $1\in C$ and $[a_2,a_3]\in E(A)$ by the assumption.
Hence the sequence $a_1,\ldots,a_n$ is a Hamilton-path of $A$ 
and there exist dashed-arrows 
$v_1 \dashrightarrow v_2$ and $v_2 \dashrightarrow v_1$ in $\widetilde{\Gamma}$ 
by Lemma~\ref{lemmaA}.
\end{proof}

\medskip

We investigate some properties of $A(n;B,C)$ and $A(n;C,B)$ in general case.

\begin{Definition}\label{def:A1b}
For $b,c \in \{1,\ldots,n-1\}$, 
let $A_1(b)$ and $A_2(c)$ be the simple graphs as $V(A_1(b))=V(A_2(c))=V(A)$, 
\begin{align*}
&E(A_1(b))=\{ \, [a_{2t-1},a_{2t-1+b}] \,|\, t\in {\mathbb{N}} \  \text{as}\  1 \le 2t-1 <2t-1+b \le n \} \ \text{and} \\
&E(A_2(c))=\{ \, [a_{2t},a_{2t+c}] \,|\, t\in {\mathbb{N}} \  \text{as}\ 2 \le 2t <2t+c \le n \}.
\end{align*}
Then $A_1(b)$ is a subgraph of $A=A(n;B,C)$ if and only if $b\in B$.
Also $A_2(c)$ is a subgraph of $A=A(n;B,C)$ if and only if $c\in C$.
Hence 
$E(A)=\bigcup_{b\in B} E(A_1(b))\cup \bigcup_{c\in C} E(A_2(c))$.
\end{Definition}

\begin{Definition}
Let $A:=A(n;B,C)$ where $n\in {\mathbb{N}}$ is a number at least $3$ 
and $B$ and $C$ are subsets of the set $\{1,\ldots,n-1 \}$.
Then we define the numbers $\beta(A)$, $\beta'(A)$, $\gamma(A)$ and $\gamma'(A)$ as follows;
\begin{align*}
&\beta(A):= \#\{ b\in B \,|\, n-b \ \text{is odd} \}, \\
&\beta'(A):= |B| - \beta(A), \\
&\gamma(A):=\#\{ c\in C \,|\, n-c \ \text{is odd} \} \ \text{and} \\
&\gamma'(A):= |C| - \gamma(A).
\end{align*}
Here we note that [$n-b$ is odd] if and only if 
[\ [$b$ is even if $n$ is odd] and [$b$ is odd if $n$ is even]\ ].
\end{Definition}

Then we have the following.

\begin{Lemma}\label{LemmaB}
Let $A=A(n;B,C)$, $A'=A(n;C,B)$, 
$V(A)=\{a_1,a_n,\ldots,a_n \}$ and $V(A')=\{a'_1,a'_2,\ldots,a'_n \}$ 
where $V(A)=\{a_1,a_n,\ldots,a_n \}$ and $V(A')=\{b_1,b_2,\ldots,b_n \}$ are 
the numbering of vertices as in the definition of $A=A(n;B,C)$ and $A'=A(n;C,B)$.
Let $a'_i:=b_{n-i+1}$ for $i=1,\ldots,n$. (Hence $V(A')=\{a'_1,a'_2,\ldots,a'_n \}$.)
Then the following two statements are equivalent$:$
\begin{enumerate}
\item[\textnormal{(i)}] $\beta(A)=\gamma(A)$.
\item[\textnormal{(ii)}] $|E(A)|=|E(A')|$, $\deg_A a_1 = \deg_{A'} a'_1$ and $\deg_A a_n = \deg_{A'} a'_n$.
\end{enumerate}
\end{Lemma}

\begin{proof}
Let $b \in \{1,\ldots,n-1\}$.
In the case that $n$ is even, 
$n-b$ is odd if and only if $b$ is odd.
Also in the case that $n$ is odd, 
$n-b$ is odd if and only if $b$ is even.

Here $n-b$ is odd 
if and only if for each $j=1,2$, 
$A_j(b)$ is symmetric and there exists the automorphism 
$\tau :A_j(b) \to A_j(b)$ as $\tau(a_i)=a_{n-i+1}$ for any $i=1,\ldots,n$.
Also $n-b$ is even if and only if for each $j=1,2$, 
$A_j(b)-a_1$ and $A_j(b)-a_n$ are symmetric and there exist the automorphisms 
$\tau' :A_j(b)-a_1 \to A_j(b)-a_1$ and 
$\tau'' :A_j(b)-a_n \to A_j(b)-a_n$ such that 
$\tau'(a_i)=a_{n-i+2}$ for any $i=2,\ldots,n$ and 
$\tau''(a_i)=a_{n-i}$ for any $i=1,\ldots,n-1$ 
(see Figure~\ref{figure-A}).

Here if $n-b$ is even then $|E(A_1(b))|=|E(A_2(b))|$ holds.
Also if $n-b$ is odd then $|E(A_1(b))|=|E(A_2(b))|+1$ holds.
Thus, for $A=A(n;B,C)$ and $A'=A(n;C,B)$, 
$|E(A)|=|E(A')|$ if and only if $\beta(A)=\beta(A')=\gamma(A)$.
Here $\beta(A')=\gamma(A)$ since $A=A(n;B,C)$ and $A'=A(n;C,B)$.

We also note that $\deg_A a_1=|B|$, $\deg_A a_n=\beta(A)+\gamma'(A)$, 
$\deg_{A'} b_1=\deg_{A'} a'_n=|C|$ and 
$\deg_{A'} b_n=\deg_{A'} a'_1=\beta(A')+\gamma'(A')=\gamma(A)+\beta'(A)$.
Then
\begin{align*}
\deg_A a_1=\deg_{A'} a'_1 \ &\Longleftrightarrow \ |B|=\gamma(A)+\beta'(A) \\
&\Longleftrightarrow \ |B|-\beta'(A)=\gamma(A) \\
&\Longleftrightarrow \ \beta(A)=\gamma(A).
\end{align*}
Also
\begin{align*}
\deg_A a_n=\deg_{A'} a'_n \ &\Longleftrightarrow \ \beta(A)+\gamma'(A)=|C| \\
&\Longleftrightarrow \ \beta(A)=|C|-\gamma'(A) \\
&\Longleftrightarrow \ \beta(A)=\gamma(A).
\end{align*}
\end{proof}

\medskip

From the above argument, we obtain the following theorem.

\begin{Theorem}\label{Thm6}
Let $\Gamma$ and $\Gamma'$ be finite simple graphs with at least three vertices 
satisfying the property~$(*)$.
If in the associated directed graph $\widetilde{\Gamma}$,
\begin{enumerate}
\item[\textnormal{(II)}] for some $v_1,v_2\in V(\widetilde{\Gamma})$, 
there exist normal-arrows $v_1 \longrightarrow v_2$ and $v_2 \longrightarrow v_1$ 
in $\widetilde{\Gamma}$
\end{enumerate}
and if the cycle $v_1,v_2$ is $\alpha$-type, 
then the graphs $\Gamma$ and $\Gamma'$ have the following structure $(\mathcal{F})${$:$}
\begin{enumerate}
\item[$(\mathcal{F})$] 
\begin{enumerate}
	\item[\textnormal{(1)}] There exists a positive number $n \ge 3$.
	\item[\textnormal{(2)}] There exists a full-subgraph $A$ of $\Gamma$ 
	with $V(A)=\{ a_1,a_2,\ldots,a_{n-1},a_n \}$ and $|V(A)|=n$.
	\item[\textnormal{(3)}] There exists a full-subgraph $A'$ of $\Gamma'$ 
	with $V(A')=\{ a'_1,a'_2,\ldots,a'_{n-1},a'_n \}$ and $|V(A')|=n$.
	\item[\textnormal{(4)}] $a_1=v_1$, $a_n=v_2$, $a'_1=v'_1$ and 
	$a'_n=v'_2$, where $v'_1:=f(v_1)$ and $v'_2:=f(v_2)$.
	\item[\textnormal{(5)}] $A=A(n;B,C)$ and $A'=A(n;C,B)$ 
	for some subsets $B$ and $C$ of the set $\{1,\ldots,n-1 \}$.
	\item[\textnormal{(6)}] $1\in C$ and $\beta(A)=\gamma(A)$ hold.
	\item[\textnormal{(7)}] If $n$ is odd 
	or if there exist dashed-arrows $v_1 \dashrightarrow v_2$ and $v_2 \dashrightarrow v_1$ 
	in $\widetilde{\Gamma}$, 
	then the sequence $a_1,\ldots,a_n$ is a Hamilton-path of $A$, 
	the sequence $a'_1,\ldots,a'_n$ is a Hamilton-path of $A'$ and $1\in B$.
	\item[\textnormal{(8)}] There exists 
	an isomorphism $\Psi:\Gamma-\{a_1,a_2 \} \to \Gamma-\{ a_{n-1},a_n \}$ 
	such that $\Psi(a_i)=a_{i-2}$ for any $i=3,\ldots,n$.
	\item[\textnormal{(9)}] There exists 
	an isomorphism $\Psi':\Gamma'-\{a'_1,a'_2 \} \to \Gamma'-\{ a'_{n-1},a'_n \}$ 
	such that $\Psi'(a'_i)=a'_{i-2}$ for any $i=3,\ldots,n$.
	\item[\textnormal{(10)}] The full-subgraph $X\cup Y$ of $\Gamma$ 
	with $V(X\cup Y)=V(\Gamma)-V(A)$ is $\Psi$-invariant.
	\item[\textnormal{(11)}] The full-subgraph $X'\cup Y'$ of $\Gamma'$ 
	with $V(X'\cup Y')=V(\Gamma')-V(A')$ is $\Psi'$-invariant.
	\item[\textnormal{(12)}] $X\cup Y$ is isomorphic to $X'\cup Y'$ 
	by $\phi_{v_1}|_{X\cup Y}$ and $\phi_{v_2}|_{X\cup Y}$ both.
	\item[\textnormal{(13)}] The full-subgraph $X$ of $\Gamma$ 
	with $V(X)=\{x\in V(X\cup Y)\,|\, \Psi(x)\neq x \}$ is 
	isomorphic to the full-subgraph $X'$ of $\Gamma'$ with 
	$V(X')=\{x'\in V(X'\cup Y')\,|\, \Psi'(x')\neq x' \}$ by 
	$\phi_{v_1}|_X$ and $\phi_{v_2}|_X$ both.
	(Here $X$ and $X'$ are the moving-parts by $\Psi$ and $\Psi'$ respectively.)
	\item[\textnormal{(14)}] The full-subgraph $Y$ of $\Gamma$ 
	with $V(Y)=\{y\in V(X\cup Y)\,|\, \Psi(y)= y \}$ is 
	isomorphic to the full-subgraph $Y'$ of $\Gamma'$ with 
	$V(Y')=\{y'\in V(X'\cup Y')\,|\, \Psi'(y')= y' \}$ by 
	$\phi_{v_1}|_Y$ and $\phi_{v_2}|_Y$ both.
	(Here $Y$ and $Y'$ are the fixed-parts by $\Psi$ and $\Psi'$ respectively.)
	\item[\textnormal{(15)}] 
	$\phi_{v_1}(a_i)=a'_{n-i+2}$ for any $i=2,\ldots,n$.
	\item[\textnormal{(16)}] 
	$\phi_{v_2}(a_i)=a'_{n-i}$ for any $i=1,\ldots,n-1$.
\end{enumerate}
\end{enumerate}
\end{Theorem}

\section{Examples and remarks on $A=A(n;B,C)$ and $A'=A(n;C,B)$}\label{sec6}

We give some examples and remarks on $A=A(n;B,C)$ and $A'=A(n;C,B)$.

Let $A=A(n;B,C)$ and $A'=A(n;C,B)$ 
for a number $n\in {\mathbb{N}}$ at least $3$ 
and subsets $B$ and $C$ of the set $\{1,\ldots,n-1 \}$.
Here $V(A)=\{a_1,a_n,\ldots,a_n \}$ is the numbering of vertices 
in the definition of $A=A(n;B,C)$.
Also if $V(A')=\{b_1,b_2,\ldots,b_n \}$ is 
the numbering of vertices in the definition of $A'=A(n;C,B)$, then 
$a'_i:=b_{n-i+1}$ for all $i=1,\ldots,n$ and $V(A')=\{a'_1,a'_2,\ldots,a'_n \}$.

We first show a proposition.

\begin{Proposition}\label{PropositionA}
Let 
\begin{align*}
&B_0:=\{ b\in B \,|\, n-b \ \text{is odd} \, \} \ \text{and}\\
&C_0:=\{ c\in C \,|\, n-c \ \text{is odd} \, \}.
\end{align*}
Then $B_0=C_0$ if and only if the following all statements hold.
\begin{enumerate}
\item[\textnormal{(1)}] $A-a_1=A-v_1$ is symmetric and there exists the automorphism 
$\tau_0 :A-a_1 \to A-a_1$ such that 
$\tau_0(a_i)=a_{n-i+2}$ for any $i=2,\ldots,n$.
\item[\textnormal{(2)}] $A-a_n=A-v_2$ is symmetric and there exists the automorphism 
$\tau_1 :A-a_n \to A-a_n$ such that 
$\tau_1(a_i)=a_{n-i}$ for any $i=1,\ldots,n-1$.
\item[\textnormal{(3)}] There exists the isomorphism $\overline{\varphi}:A\to A'$ 
such that $\overline{\varphi}(a_i)=a'_i$ for any $i=1,\ldots,n$.
Hence $A \cong A'$.
\item[\textnormal{(4)}] The isomorphisms $\varphi_0:=\phi_{v_1}\circ\tau_0:A-v_1 \to A'-v'_1$ 
and $\varphi_1:=\phi_{v_2}\circ\tau_1:A-v_2 \to A'-v'_2$ both 
extend to the isomorphism $\overline{\varphi}:A\to A'$ in $(3)$.
\end{enumerate}
\end{Proposition}

\begin{proof}
Suppose that $B_0=C_0$.

Then for any $b\in B_0$, since $b\in C_0$, 
the graphs $A_1(b)$ and $A_2(b)$ (in Definition~\ref{def:A1b}) 
both are subgraphs of $A$.
Let $A_1(b) \cup A_2(b)$ be the simple graph as 
$V(A_1(b) \cup A_2(b))=V(A)$ 
and $E(A_1(b) \cup A_2(b))=E(A_1(b)) \cup E(A_2(b))$.
Then $A_1(b) \cup A_2(b)$, $A_1(b) \cup A_2(b)-a_1$ and $A_1(b) \cup A_2(b)-a_n$ 
are symmetric in the above sense.

Also by the argument in the proof of Lemma~\ref{LemmaB}, 
for any $b\in B-B_0$ ($n-b$ is even) and for any $j=1,2$, 
$A_j(b)-a_1$ and $A_j(b)-a_n$ are symmetric in the above sense.
Hence the statements (1) and (2) hold.

Now we show that the statements (3) and (4) hold.
For any $i=2,\ldots,n$,
\[ \varphi_0(a_i)=\phi_{v_1}(\tau_0(a_i))
=\phi_{v_1}(a_{n-i+2})= a'_i=\overline{\varphi}(a_i). \]
Also for any $i=1,\ldots,n-1$,
\[ \varphi_1(a_i)=\phi_{v_2}(\tau_1(a_i))
=\phi_{v_2}(a_{n-i})= a'_i=\overline{\varphi}(a_i). \]
Here $a_1=v_1$ and $a_n=v_2$ (and $a'_1=v'_1$ and $a'_n=v'_2$)
do not span an edge of $A$ (and $A'$ respectively).

Hence for any $i,j \in \{1,\ldots,n\}$,
$[a_i,a_j]\in E(A)$ if and only if $[a'_i,a'_j]\in E(A')$, 
because $\varphi_0$ and $\varphi_1$ are isomorphisms.
Thus $\overline{\varphi}:A\to A'$ as 
$\overline{\varphi}(a_i)=a'_i$ for $i=1,\ldots,n$ 
is an isomorphism and the statements (3) and (4) hold.

Conversely we show that if the statements (1) and (2) hold 
then $B_0=C_0$.
Suppose that $B_0\neq C_0$.
Then there exists $b_0 \in B_0-C_0$ or $c_0 \in C_0-B_0$.
If there is $b_0 \in B_0-C_0$, then 
$A_1(b_0)$ is a subgraph of $A$ and 
$A_2(b_0)$ is not a subgraph of $A$.
If there is $c_0 \in C_0-B_0$, then 
$A_2(c_0)$ is a subgraph of $A$ and 
$A_1(c_0)$ is not a subgraph of $A$.
Hence $A-a_1$ and $A-a_n$ are not symmetric in the above sense.
Thus (1) and (2) do not hold.
\end{proof}

We remark that 
any $A=A(n;B,C)$ and $A'=A(n;C,B)$ with $1\in C$ and $B_0=C_0$ 
as in Proposition~\ref{PropositionA} are isomorphic 
and they can be realized 
as the full-subgraphs $A$ and $A'$ of $\Gamma$ and $\Gamma'$ respectively 
in the structure $(\mathcal{F})$.
On the other hand, if $B_0\neq C_0$ then in general 
there is a possibility that 
$A=A(n;B,C)$ and $A'=A(n;C,B)$ 
can not become finite simple graphs satisfying the condition~$(*)$.
We introduce this in Example~\ref{example-A5}.
Also in the case that $B_0\neq C_0$, 
even if $A=A(n;B,C)$ and $A'=A(n;C,B)$ 
are finite simple graphs satisfying the property~$(*)$, 
it is not obvious that $A\cong A'$.
This can be seen from Examples~\ref{example-A1}--\ref{example-A4}.

\begin{Example}\label{example-A1}
We consider $A=A(8;\{3\},\{1\})$ and $A'=A(8;\{1\},\{3\})$ 
(see Figure~\ref{figure-A-example1}).
Here $n=8$ is even.
Then $A$ and $A'$ have two connected components and they are isomorphic.

\begin{figure}[h]
\centering
{
\vspace*{3mm}
\includegraphics[keepaspectratio, scale=0.90, bb=0 0 293 81]{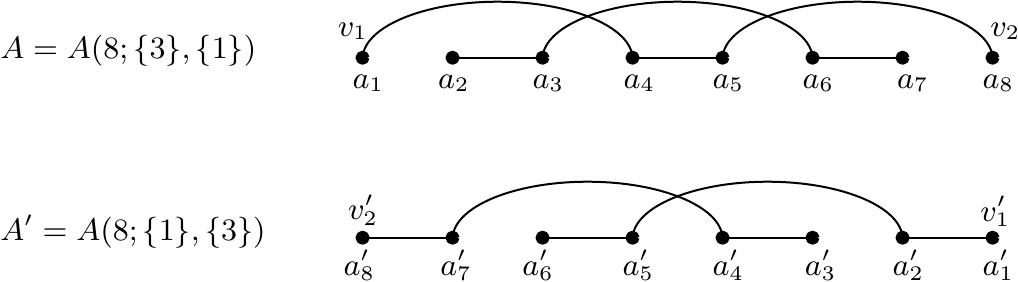}
}
\caption{Example~\ref{example-A1}}\label{figure-A-example1}
\end{figure}

In this example, since $A\cong A'$, 
we can construct a bijective map $f:V(A)\to V(A')$ and 
an isomorphism $\phi_v:A-v \to A'-f(v)$ for any $v\in V(A)$.
Here for $v_1=a_1$, $v_2=a_8$, $v'_1=f(v_1)=a'_1$ and $v'_2=f(v_2)=a'_8$, 
we can take the isomorphisms $\phi_{v_1}:A-v_1 \to A'-v'_1$ and 
$\phi_{v_2}:A-v_2 \to A'-v'_2$ as (14) and (15) in Theorem~\ref{Thm6} 
and for $\Gamma:=A$ and $\Gamma':=A'$,
there exist normal-arrows $v_1 \longrightarrow v_2$ and $v_2 \longrightarrow v_1$ 
in $\widetilde{\Gamma}$.
Then there do not exist dashed-arrows $v_1 \dashrightarrow v_2$ and $v_2 \dashrightarrow v_1$ 
in $\widetilde{\Gamma}$, 
since the sequence $a_1,a_2,\ldots,a_8$ is not a Hamilton-path of $A$.
\end{Example}

\begin{Example}\label{example-A2}
We consider $A=A(8;\{1,3\},\{1,5\})$ and $A'=A(8;\{1,5\},\{1,3\})$ 
(see Figure~\ref{figure-A-example2}).
Here $n=8$ is even.
Then $A$ and $A'$ are isomorphic as in Figure~\ref{figure-A-example2}.

\begin{figure}[h]
\centering
{
\vspace{3mm}
\includegraphics[keepaspectratio, scale=0.90, bb=0 0 301 220]{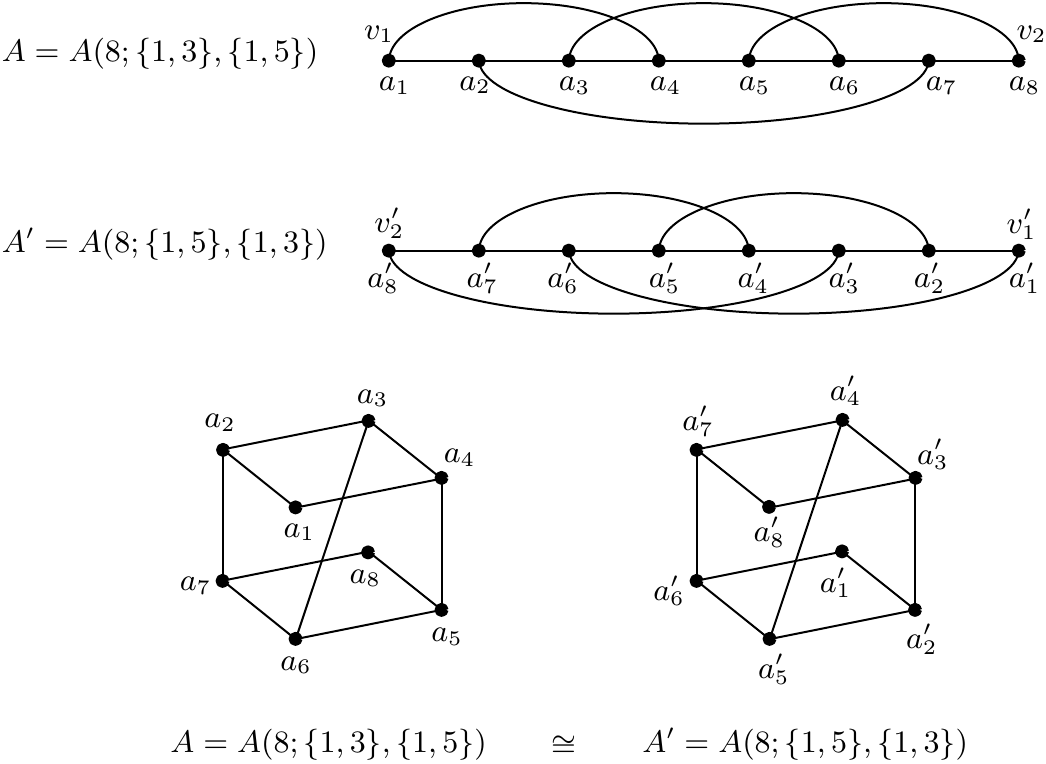}
}
\caption{Example~\ref{example-A2}}\label{figure-A-example2}
\end{figure}

In this example, since $A\cong A'$, 
we can construct a bijective map $f:V(A)\to V(A')$ and 
an isomorphism $\phi_v:A-v \to A'-f(v)$ for any $v\in V(A)$ 
such that for $v_1=a_1$, $v_2=a_8$, $v'_1=f(v_1)=a'_1$, $v'_2=f(v_2)=a'_8$ and 
for $\Gamma:=A$ and $\Gamma':=A'$, 
there exist arrows $v_1 \longrightarrow v_2$, $v_2 \longrightarrow v_1$, 
$v_1 \dashrightarrow v_2$ and $v_2 \dashrightarrow v_1$ 
in $\widetilde{\Gamma}$.
\end{Example}

\begin{Example}\label{example-A3}
We consider $A=A(8;\{3,4\},\{1\})$ and $A'=A(8;\{1\},\{3,4\})$ 
(see Figure~\ref{figure-A-example3}).
Here $n=8$ is even 
and there are an even number and an odd number both in $B=\{3,4\}$.
Then $A$ and $A'$ have two connected components and they are isomorphic 
as in Figure~\ref{figure-A-example3}.

\begin{figure}[h]
\centering
{
\vspace{5mm}
\includegraphics[keepaspectratio, scale=0.90, bb=0 0 298 210]{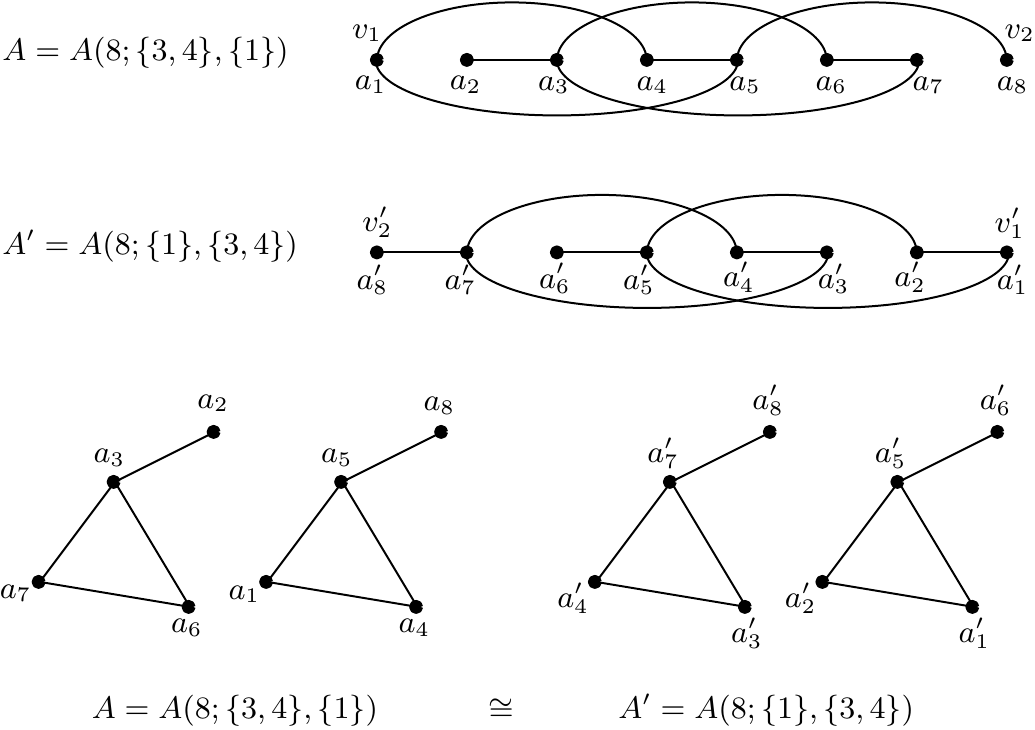}
}
\caption{Example~\ref{example-A3}}\label{figure-A-example3}
\end{figure}

In this example, since $A\cong A'$, 
we can construct a bijective map $f:V(A)\to V(A')$ and 
an isomorphism $\phi_v:A-v \to A'-f(v)$ for any $v\in V(A)$ 
such that for $v_1=a_1$, $v_2=a_8$, $v'_1=f(v_1)=a'_1$, $v'_2=f(v_2)=a'_8$ and 
for $\Gamma:=A$ and $\Gamma':=A'$,
there exist normal-arrows $v_1 \longrightarrow v_2$ and $v_2 \longrightarrow v_1$ 
in $\widetilde{\Gamma}$.
Here there do not exist dashed-arrows $v_1 \dashrightarrow v_2$ and $v_2 \dashrightarrow v_1$ 
in $\widetilde{\Gamma}$.
\end{Example}

\begin{Example}\label{example-A4}
We consider $A=A(7;\{1,2,3\},\{1,4\})$ and $A'=A(7;\{1,4\},\{1,2,3\})$ 
(see Figure~\ref{figure-A-example4}).
Here $n=7$ is odd and 
there are an even number and an odd number both in $B-C=\{2,3\}$.
Then $A$ and $A'$ are isomorphic as in Figure~\ref{figure-A-example4}.

\begin{figure}[h]
\centering
{
\vspace{3mm}
\includegraphics[keepaspectratio, scale=0.90, bb=0 0 289 224]{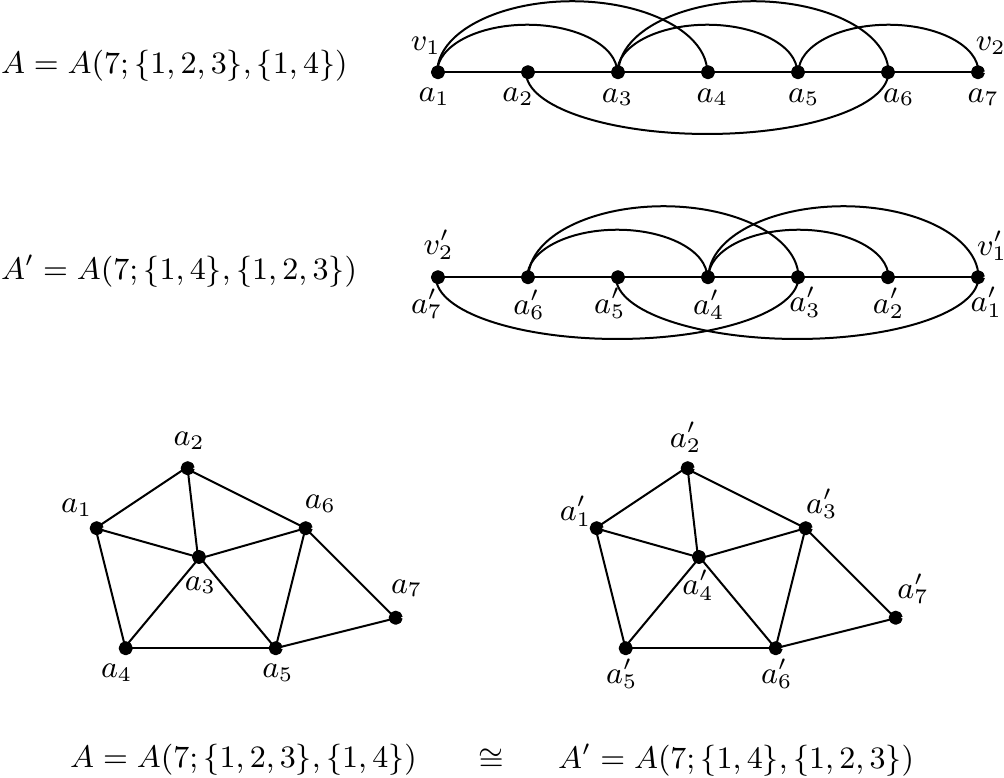}
}
\caption{Example~\ref{example-A4}}\label{figure-A-example4}
\end{figure}

In this example, since $A\cong A'$, 
we can construct a bijective map $f:V(A)\to V(A')$ and 
an isomorphism $\phi_v:A-v \to A'-v'$ for any $v\in V(A)$ 
such that for $v_1=a_1$, $v_2=a_7$, $v'_1=f(v_1)=a'_1$, $v'_2=f(v_2)=a'_7$ and 
for $\Gamma:=A$ and $\Gamma':=A'$,
there exist arrows $v_1 \longrightarrow v_2$, $v_2 \longrightarrow v_1$, 
$v_1 \dashrightarrow v_2$ and $v_2 \dashrightarrow v_1$
in $\widetilde{\Gamma}$.
\end{Example}

\begin{Example}\label{example-A5}
We consider $A=A(8;\{1,3,4\},\{1,5\})$ and $A'=A(8;\{1,5\},\{1,3,4\})$ 
(see Figure~\ref{figure-A-example5}).
Here $\beta(A)=\gamma(A)$ holds.
Also we note that 
there exists a bijective map $g:V(A)\to V(A')$ such that $\deg_A v=\deg_{A'} g(v)$ 
for all $v\in V(A)$.

\begin{figure}[h]
	\centering
	{
		\vspace{4mm}
		\includegraphics[keepaspectratio, scale=0.90, bb=0 0 316 104]{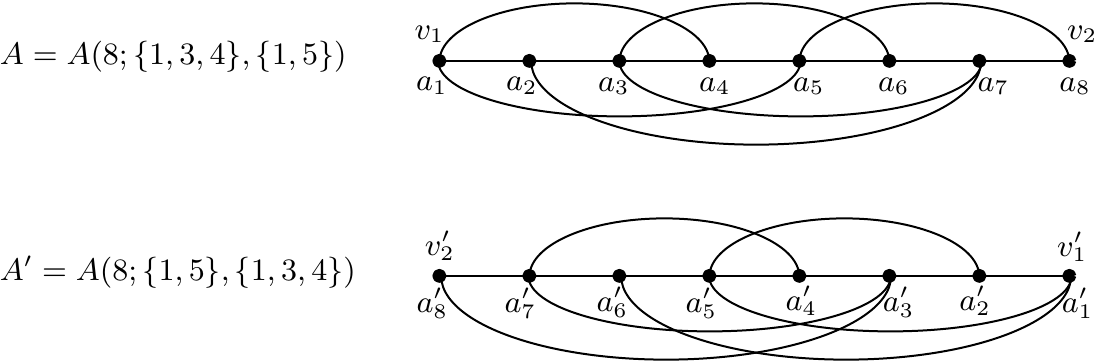}
	}
	\caption{Example~\ref{example-A5}}\label{figure-A-example5}
\end{figure}

Then $A$ and $A'$ are not isomorphic.
Indeed, there is a triangle $a'_3,a'_7,a'_8$ in $A'$ 
such that $\deg_{A'}a'_3=4$, $\deg_{A'}a'_7=4$ and $\deg_{A'}a'_8=2$.
On the other hand, there is not such a triangle in $A$.
In the graph $A$, $a_8$ is the unique vertex with degree $2$ 
and $a_5$ and $a_7$ are the vertices connecting to $a_8$ with degree $4$.
Here $a_5$ and $a_7$ do not span an edge in $A$.

It is known that 
every simple graph $\Gamma$ with at least 3 vertices and at most 11 vertices 
is reconstructible \cite{Mc2}.
Hence the two graphs $A=A(8;\{1,3,4\},\{1,5\})$ and $A'=A(8;\{1,5\},\{1,3,4\})$ 
can not satisfy the property~$(*)$.
\end{Example}

By Example~\ref{example-A5}, in general, 
there is a possibility that 
$A=A(n;B,C)$ and $A'=A(n;C,B)$ 
(still if they satisfy $\beta(A)=\gamma(A)$ and 
still if there is a bijective map $g:V(A)\to V(A')$ such that $\deg_A v=\deg_{A'} g(v)$ 
for all $v\in V(A)$) 
can not become finite simple graphs satisfying the property~$(*)$.

\medskip

Here the following problem arises.

\begin{Problem}
The full-subgraphs $A=A(n;B,C)$ and $A'=A(n;C,B)$ of $\Gamma$ and $\Gamma'$ respectively 
in Theorem~\ref{Thm6} will be isomorphic?
\end{Problem}

\section{Examples and remarks on $\beta$-type cycles}\label{sec7}

We give an example and a remark on $\beta$-type cycles.

\begin{Example}\label{example-B1}
We consider two graphs $\Gamma$ and $\Gamma'$ as in Figure~\ref{figure-B-example1}.
Here $\Gamma$ and $\Gamma'$ are isomorphic.
We define isomorphisms $\phi_{v_1}:\Gamma -v_1 \to \Gamma' -v'_1$ 
and $\phi_{v_2}:\Gamma -v_2 \to \Gamma' -v'_2$ as in Figure~\ref{figure-B-example1-2}.

\begin{figure}[h]
	\centering
	{
		\vspace{4mm}
		\includegraphics[keepaspectratio, scale=0.90, bb=0 0 253 210]{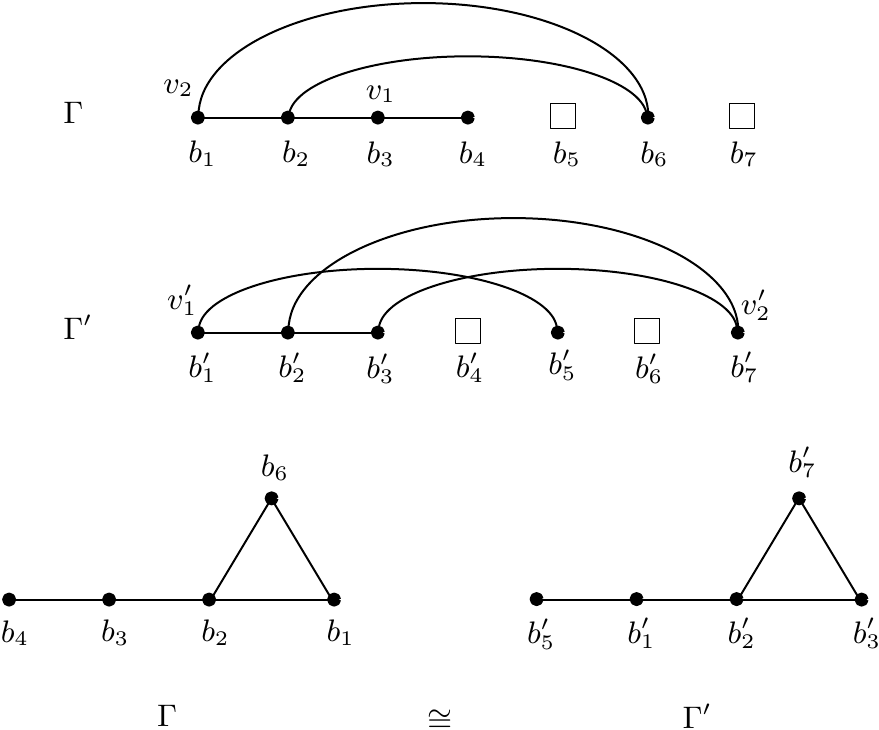}
	}
	\caption{Example~\ref{example-B1}}\label{figure-B-example1}
\end{figure}

We consider a sequence $\{b'_i\}$ in $V(\Gamma)$ as 
\begin{align*}
&b'_1:=v'_1, \\
&b'_2:=\phi_{v_1}(v_2), \\
&b'_3:=\phi_{v_1}\circ \phi^{-1}_{v_2}(v'_1), \\
&b'_4:=\phi_{v_1}\circ \phi^{-1}_{v_2}\circ\phi_{v_1}(v_2), \\
&b'_5:=\phi_{v_1}\circ \phi^{-1}_{v_2}\circ\phi_{v_1}\circ \phi^{-1}_{v_2}(v'_1), \\
&\cdots.
\end{align*}
Here $b'_4$, $b'_6$ and $b'_i$ ($i \ge 8$) are blanks.

Also consider a sequence $\{b_i\}$ in $V(\Gamma)$ as 
\begin{align*}
&b_1:=v_2, \\
&b_2:=\phi^{-1}_{v_2}(v'_1), \\
&b_3:=\phi^{-1}_{v_2}\circ \phi_{v_1}(v_2), \\
&b_4:=\phi^{-1}_{v_2}\circ \phi_{v_1}\circ\phi^{-1}_{v_2}(v'_1), \\
&b_5:=\phi^{-1}_{v_2}\circ \phi_{v_1}\circ\phi^{-1}_{v_2}\circ \phi_{v_1}(v_2), \\
&\cdots.
\end{align*}

\begin{figure}[h]
	\centering
	{
		\includegraphics[keepaspectratio, scale=0.90, bb=0 0 427 300]{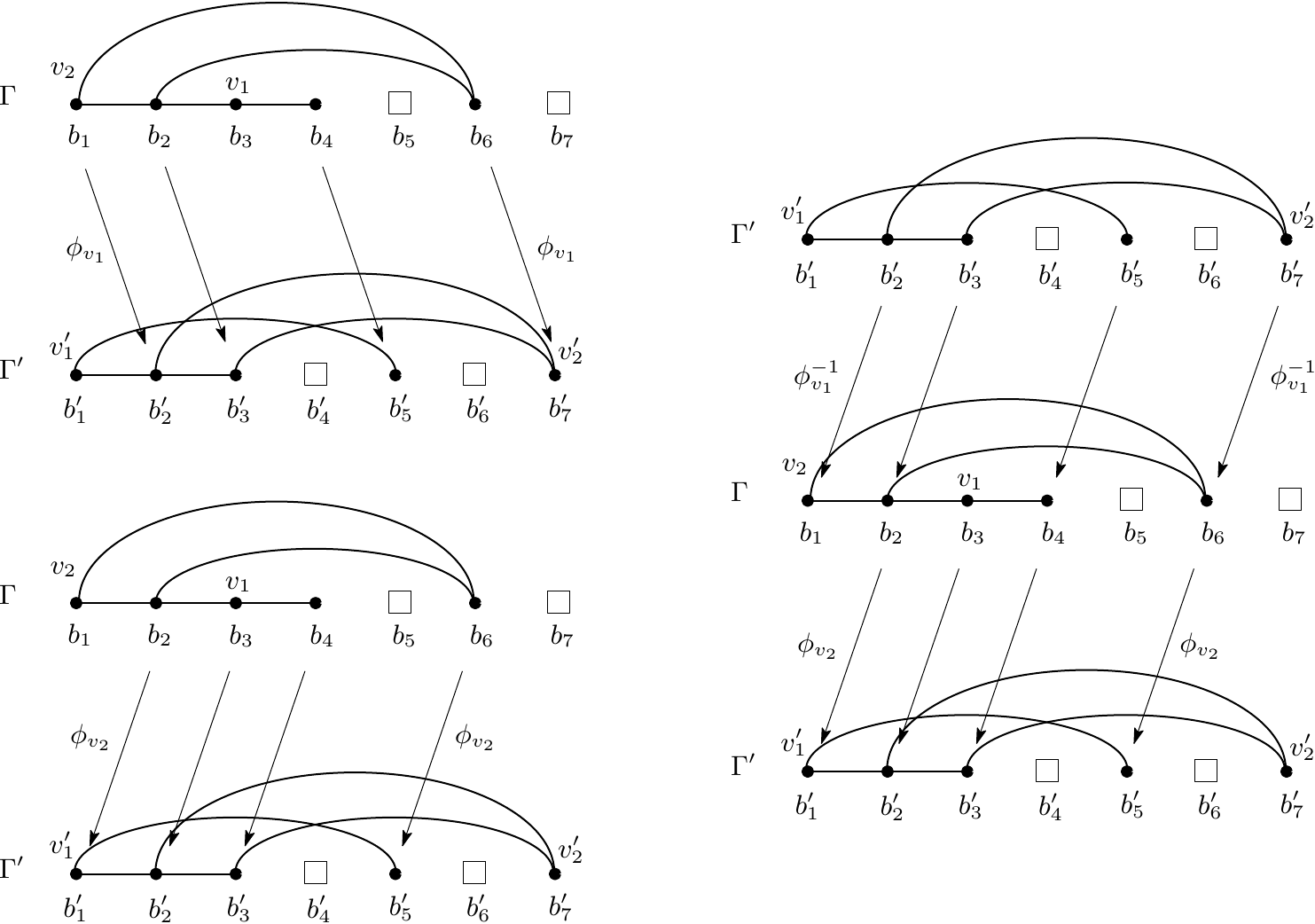}
\vspace*{1mm}
	}
	\caption{Isomorphisms $\phi_{v_1}$ and $\phi_{v_2}$}\label{figure-B-example1-2}
\end{figure}

Here $b_5$, $b_7$ and $b_i$ ($i \ge 8$) are blanks.

Then $[v'_1,\phi_{v_1}(v_2)]=[b'_1,b'_2]\in E(\Gamma')$, 
$[v'_2,\phi_{v_2}(v_1)]=[b'_7,b'_2]\in E(\Gamma')$, 
$[v_1,\phi^{-1}_{v_1}(v'_2)]=[b_3,b_6] \not\in E(\Gamma)$ 
and $[v_2,\phi^{-1}_{v_2}(v'_1)]=[b_1,b_2] \in E(\Gamma)$.
Hence there exist arrows $v_1 \longrightarrow v_2$, $v_2 \longrightarrow v_1$ 
and $v_2 \dashrightarrow v_1$ in $\widetilde{\Gamma}$.
The cycle $v_1,v_2$ is $\beta$-type.

The isomorphism 
\[ \phi_{v_2} \circ \phi^{-1}_{v_1}:\Gamma' -\{b'_1,b'_2\} \xrightarrow[\cong]{\; \phi^{-1}_{v_1}| \;} 
\Gamma -\{v_1,v_2\} \xrightarrow[\cong ]{\; \phi_{v_2}| \;} \Gamma' -\{b'_2,b'_7\} \]
is as in Figure~\ref{figure-B-example1-2}, 
where $b'_1=v'_1$, $b'_2=\phi_{v_1}(v_2)=\phi_{v_2}(v_1)$ and $b'_7=v'_2$.
\end{Example}

\begin{Remark}
We define a subgraph $\overline{A}(n,p;B,C)$ of $A(n;B,C)$.

Let $n\in {\mathbb{N}}$ be a number at least $3$, 
let $p\in {\mathbb{N}}$ with $p < n$ 
and let $B$ and $C$ be subsets of the set $\{1,\ldots,n-1 \}$.
Let $\{ a_1,\ldots,a_n \}$ be the vertex set of $A(n;B,C)$ 
whose numbering is as in the definition of $A(n;B,C)$.
Let $P:= \{ a_i \,|\, i=p+2t \ (0 \le t \in {\mathbb{Z}}) \}$.

Then we define $\overline{A}(n,p;B,C)$ as the full-subgraph 
of $A(n;B,C)$ whose vertex set is $\{ a_1,\ldots,a_n \}-P$; 
that is, $a_i$ is blank if $a_i\in P$.

In Example~\ref{example-B1}, 
the graphs $\Gamma$ and $\Gamma'$ 
with the vertex sets $\{ b_1,b_2,b_3,b_4,b_5,b_6,b_7 \}$ 
and $\{ b'_1,b'_2,b'_3,b'_4,b'_5,b'_6,b'_7 \}$ respectively 
can be denoted by $\Gamma=\overline{A}(7,5;\{1,5\},\{1,4\})$ 
and $\Gamma'=\overline{A}(7,4;\{1,4\},\{1,5\})$.
\end{Remark}

%

%
\end{document}